\documentclass{article}

\usepackage{amsmath,amsthm,amssymb}
\usepackage{fullpage}
\usepackage{graphics}
\usepackage{environ}
\usepackage{url}
\usepackage{algorithm}
\usepackage[noend]{algpseudocode}
\usepackage[labelfont=bf]{caption}
\usepackage{cite}
\usepackage{framed}
\usepackage[framemethod=tikz]{mdframed}
\usepackage{appendix}
\usepackage{graphicx}
\usepackage[textsize=tiny]{todonotes}
\usepackage{mathtools}

\newtheorem{theorem}{Theorem}[section]
\newtheorem*{theorem*}{Theorem}
\newtheorem{claim}[theorem]{Claim}
\newtheorem{lemma}[theorem]{Lemma}
\newtheorem{define}[theorem]{Definition}
\newtheorem*{define*}{Definition}
\newtheorem{definition}[theorem]{Definition}
\newtheorem*{definition*}{Definition}
\newtheorem{corollary}[theorem]{Corollary}

\newtheorem*{remark}{Remark}
\newtheorem{notation}[theorem]{Notation}
\newtheorem{conjecture}[theorem]{Conjecture}
\newtheorem{fact.}[theorem]{Fact}

\newcommand\F{{\mathbb{F}}}
\newcommand\K{{\mathbb{K}}}

\newcommand\Q{{\mathbb{Q}}}
\newcommand\Z{{\mathbb{Z}}}
\newcommand\N{{\mathbb{N}}}

\newcommand\C{{\mathbb{C}}}

\newcommand\perm{\textsf{perm}}
\newcommand\ord{\textsf{ord}}
\newcommand\rep{\textsf{rep}}

\newcommand\mult{\textsf{mult}}
\newcommand\red{\textsf{red}}

\newcommand\rank{\textsf{rank}}

\newcommand\fact{\textsf{fact}}

\DeclareMathOperator{\lcm}{lcm}

\DeclareMathOperator{\Gal}{Gal}  %

\newcommand{\ts}{\textsuperscript}
\newcommand{\eps}{\epsilon}

\newcommand{\ignore}[1]{}

\DeclareMathOperator{\DFT}{DFT}

\newcommand{\expref}[2]{#1 \ref{#2}}

\begin{document}

\title{Fourier and Circulant Matrices are Not Rigid} 
\author{Zeev Dvir\thanks{Research supported by NSF CAREER award DMS-1451191 and NSF grant CCF-1523816.} \and Allen Liu}
\date{}

\maketitle
\begin{abstract}
The concept of \emph{matrix rigidity}
was first introduced by Valiant in 1977.
Roughly speaking, a matrix is rigid if its rank cannot be reduced 
significantly by changing a small number of entries.  There has been
considerable interest in  the explicit construction of 
rigid matrices as Valiant showed in his MFCS'77 paper that
explicit families of rigid matrices  can be used to prove
lower bounds for arithmetic circuits.

In a surprising recent result, Alman and Williams (FOCS'19)
showed that the $2^n\times 2^n$ Walsh--Hadamard matrix,  
which was conjectured to be rigid, is actually not very rigid.
This line of work was extended by 
Dvir and Edelman (\emph{Theory of Computing}, 2019) 
to a family of matrices related to the Walsh--Hadamard matrix, 
but over finite fields.  In the present paper we take another
step in this direction and show that for any abelian group $G$ and function $f: G \rightarrow \C$,
the \emph{$G$-circulant matrix,}
given by $M_{xy} = f(x - y)$ for $x,y \in G$, is not rigid over $\C$.  Our results also hold if we replace $\C$ with a finite field $\F_q$ and require that $\gcd(q,|G|) = 1$. 
En route to our main result, we show  
that circulant and Toeplitz matrices (over finite fields or $\C$) and
Discrete Fourier Transform (DFT)  matrices (over $\C$) are
not sufficiently rigid %
to carry out Valiant's approach to
proving circuit lower bounds.  
This complements a recent result of Goldreich and Tal
(\emph{Comp. Complexity}, 2018)  
who showed that Toeplitz matrices are nontrivially rigid (but not enough
for Valiant's method).  Our work differs from previous non-rigidity
results in that those papers 
considered matrices whose underlying group of symmetries was of the form 
$\Z_p^n$ with $p$ fixed and $n$ tending to infinity, while in the families
of matrices we study, the underlying group of symmetries can be any abelian group and, in particular, the cyclic group $\Z_N$, which has very different structure. Our results also suggest natural new candidates for rigidity in the form of matrices whose symmetry groups are highly non-abelian.
\end{abstract}
\newpage

\tableofcontents
\newpage

\section{Introduction}

\subsection{Background}
A major goal in complexity theory is to prove lower bounds on the size and depth of arithmetic circuits that compute certain functions.  One specific problem that remains open despite decades of effort is to find functions for which we can show super-linear size lower bounds for circuits of logarithmic depth.  In \cite{rigidity}, Valiant introduced the notion of matrix rigidity as a possible method of proving such lower bounds for arithmetic circuits.  More precisely, over a field $\F$, an $m \times n$ matrix $M$  is said to be $(r,s)$-rigid if any $m \times n$ matrix of rank at most $r$ differs from $M$ in at least $s$ entries.  Valiant showed that for any linear function $f:\F^n \rightarrow \F^n$ that can be computed by an arithmetic circuit of size $O(n)$ and depth $O(\log n)$, the corresponding matrix can be reduced to rank $O(\frac{n}{\log \log n})$ by changing $O(n^{1+\epsilon})$ entries for any $\epsilon > 0$.  Thus, to prove a circuit lower bound for a function $f$, it suffices to lower bound the rigidity of the corresponding matrix at rank $O(\frac{n}{\log \log n})$.  We call a matrix Valiant-rigid if it is 
$\left( O(\frac{n}{\log \log n}), \Omega(n^{1+\epsilon}) \right)$-rigid
for some $\epsilon > 0$, i.e., sufficiently rigid for Valiant's method to yield circuit lower bounds. Over any infinite field, Valiant shows that almost all $n \times n$ matrices are $(r, (n-r)^2)$-rigid for any $r$, while over a finite field one can get a similar result with a logarithmic loss in the sparsity parameter.  Despite much effort,
explicit constructions of rigid matrices have remained elusive.  

Over infinite (or very large) fields, there are ways to construct  highly rigid matrices using either algebraically independent  entries or entries that have exponentially large description (see \cite{construction1, construction2, Vandermonde}).\footnote{It remains open to construct a matrix
 that is Valiant-rigid, even if we only require that the entries live
 in a number field of dimension polynomial in the size of the matrix.}  
However, these constructions are not considered to be fully explicit as they do not tell us anything about the computational complexity of the corresponding function.  Ideally, we would be able to construct rigid $(0,1)$-matrices, 
but even a construction where the entries are in a reasonably simple field
(such as the $m$\ts{th} cyclotomic field for a small value of $m$)  %
would be a major breakthrough.  The best known constructions of such
matrices are 
$(r,\Omega(\frac{n^2}{r} \log \frac{n}{r}))$-rigid
(see \cite{rigidity_construction, friedman_rigidity}).  There has also been work towards constructing semi-explicit rigid matrices.  Semi-explicit constructions which require $O(n)$ bits of randomness (instead of the usual $O(n^2)$) would still yield circuit lower bounds through Valiant's approach.\footnote{Note however, that it is easy to construct rigid matrices with $O(n^{1 + \epsilon})$ bits of randomness for any $\epsilon > 0$ (for example by taking a random matrix with at most $n^{\epsilon}$ non-zeros per row) but this is not sufficient for Valiant's approach.}  The best result in this realm (see \cite{toeplitz}) shows that random Toeplitz matrices are $(r, \frac{n^3}{r^2 \log n})$-rigid with high probability for $r \geq \Omega(\sqrt n)$.

Note that both of these bounds become trivial when $r$ is
$n/\log\log n$.       %
Other variants of semi-explicit constructions have also been studied.
\cite{npconstruction} gives a construction of 
$(2^{(\log n)^{1/4 - o(1)}}, \Omega(n^2) )$-rigid 
matrices using an NP-oracle.  This construction is not in the regime
for Valiant-rigidity.

Many well-known families of matrices, such as Hadamard matrices
(square matrices with $\pm 1$ entries whose rows are orthogonal) and DFT
(Discrete Fourier Transform)  matrices, have been conjectured to be 
Valiant-rigid \cite{survey}.  However, a recent line of work
(see \cite{old_hadamard, CLP}) shows that certain well-structured matrices
are not rigid.  Alman and Williams show in \cite{old_hadamard} that the
Walsh--Hadamard matrix, i.e., the $2^n \times 2^n$ Hadamard matrix given by 
$H_{xy} = (-1)^{\langle x , y \rangle}$ 
as $x$ and $y$ range over $\{ 0,1 \}^n$, is not Valiant-rigid over $\Q$.
Along similar lines, Dvir and Edelman show in \cite{CLP} that 
$G$-circulant matrices for the additive group of
$\F_p^n$, given by $M_{xy} = f(x-y)$ where $f: \F_p^n \rightarrow \F_p$ and $x,y$ range over $\F_p^n$, are not Valiant-rigid over $\F_p$ (where we view $p$ as fixed and $n$ goes to infinity).  The Walsh--Hadamard 
matrix and the $G$-circulant matrices for the additive group of  
$\F_p^n$ have 
the property that for any $\epsilon > 0$, there exists an $\epsilon'>0$ such that it is possible to change at most $N^{1 + \epsilon}$ entries and reduce the rank to $N^{1 - \epsilon'}$ (where $N$ denotes the size of the matrix).  The proofs of both results rely on constructing a matrix determined by a polynomial $P(x,y)$ that agrees with the given matrix on almost all entries and then arguing that the constructed matrix has low rank.   

\subsection{Our contribution} 
\begin{definition}[$G$-circulant matrices]
Let $G$ be a finite abelian group, $\F$ a field, and $f : G\to\F$
a function.  The $G$-circulant matrix $M(f)$ is defined as the
$|G|\times |G|$ matrix whose rows and columns are labeled by
the elements of $G$ and whose $(x,y)$ entry is $f(x-y)$
(for $x,y\in G$).
\end{definition}   %
In this paper we prove that for an abelian group $G$, over any finite field with characteristic relatively prime to $|G|$
and over the complex numbers, $G$-circulant matrices %
are not Valiant-rigid (Theorem \ref{allabelian_intro}). 
En route to our main result, we prove that the following
commonly studied families of matrices are not Valiant-rigid:
\begin{itemize}
\item  
DFT matrices (over $\C$).   %
    \item Circulant matrices (over finite fields and $\C$): matrices whose rows are obtained by cyclically shifting the top row.
    \item Toeplitz matrices (over finite fields and $\C$): matrices with constant diagonals.\footnote{It is not hard to see that rigidity of circulant and Toeplitz matrices is essentially the same question so for the sake of consistency with our (group theoretic)  approach we will primarily consider circulant matrices.}
\end{itemize}
\begin{remark}
For circulant and Toeplitz matrices over finite fields, we do not require any additional conditions, i.e. we do not require that the size of the matrix and the characteristic of the field are relatively prime.  See the beginning of Section \ref{sec_groupalgebra_finitefield} for a more in-depth discussion about why we require such a condition for general abelian groups.
\end{remark}
The families of matrices we consider in our paper  
have very different underlying group structure than those
considered in previous work.  Both
\cite{old_hadamard} and \cite{CLP}
analyze matrices constructed from an underlying group of the form
$\Z_p^n$ with $p$ a fixed prime number and $n$ tending to infinity.  
In this paper we study matrices whose underlying symmetry group can
be any abelian group.  In fact, the core of our proof is handling the
case when the underlying group is cyclic.

Circulant matrices are the special case of
$G$-circulants
for cyclic groups $G$.  Similarly, the DFT matrices 
are the special case of the $\DFT_G$ matrices %
(where $\DFT_G$ is the matrix given by the character table 
of an abelian group $G$) when
$G$ is cyclic.
The Walsh--Hadamard matrices are another special case of
the $\DFT_G$ matrices, where $G$ is the group $\Z_2^n$.  %
We use the fact, that 
every finite abelian group can be decomposed into 
the direct product of cyclic groups, 
to extend our results to all abelian groups, although this
extension is by no means immediate.  
While most natural constructions of matrices are highly symmetric,
our results %
show  %
that matrices that are  symmetric under abelian groups are not rigid and that perhaps we should look toward less structured matrices, or matrices whose symmetry group is non-abelian, as candidates for rigidity.   

We now move into a more technical overview of our paper. %
We begin with a few definitions. 
\begin{definition}
Define the \emph{regular-rigidity} 
$\textsf{r}^{\F}_A(r)$  
of a matrix $A$ over a field $\F$  %
as the minimum value of $s$ such that it is possible to change at most $s$ 
entries in each row and column of $A$ to obtain a matrix of rank at most $r$.
\end{definition}
When the field is clear from context, we will omit the superscript.%
The notion of regular-rigidity is weaker than the usual notion of rigidity (and is also weaker than the commonly used notion of row-rigidity) as if $A$ is an $n \times n$ matrix and $A$ is $(r,ns)$-rigid then $\textsf{r}_A(r) \geq s$.  Note that this actually makes our results stronger as we will show that the matrices we consider are not regular-rigid.

To simplify the exposition, we define a
qualitative notion  %
of non-rigidity we call QNR (quasipolynomial non-rigidity).
\begin{define}
We say that
a family $\cal A$ of matrices %
is \emph{quasipolynomially non-rigid} (QNR)  %
over a field $\F$ if there are constants $c_1,c_2 > 0$ such that for
any $\eps > 0$, all sufficiently large matrices $M \in \cal A$ satisfy
\[
\textsf{r}^{\F}_M\left( \frac{N}{\exp\left(\eps^{c_1}(\log N)^{c_2}\right)}
 \right) \leq N^{\eps} \,,
\]
where $M$ is an $N \times N$ matrix.
\end{define}  
We will prove that various families of matrices are QNR.
Note that this immediately implies that they are not Valiant-rigid.  Our main results are stated below.

\begin{theorem}\label{allabelian_intro}
Let $G$ be an abelian group.  The family of $G$-circulant matrices is QNR over $\C$.  For a finite field $\F_q$, if $\gcd(|G|,q) = 1$, then the family of $G$-circulant matrices is QNR over $\F_q$.
\end{theorem}

The formal statement and proof of this result can be found in Section \ref{sec_groupalgebra} (see Theorem \ref{allabelian}) for the complex numbers and Section \ref{sec_groupalgebra_finitefield} (see Theorem \ref{finitefield_allabelian}) for finite fields.
\begin{theorem}\label{thm:extension_degree_intro}
Let $G$ be an abelian group of order $N$.  Then there exists
$m = \tilde{O}(N^3)$, depending only on $G$, such that the rational
$G$-circulant matrices are QNR over the $m$\ts{th}
cyclotomic field.  The same also holds for the matrix $\DFT_G$.
\end{theorem}
The formal statement and proof of this result can be found in Section \ref{sec_groupalgebra} (see Theorem \ref{thm:extension_degree}).

In addition to the aforementioned results for circulant, Toeplitz and
DFT matrices, 
our main theorem has a few more consequences that are worth mentioning.
The following two corollaries to our main result were pointed out 
by Babai and Kivva~\cite{kivva}:
\begin{itemize}
\item  The Paley--Hadamard matrices are QNR over $\C$.
\item  The Vandermonde matrices $V_n(x_1, \dots, x_n)$ whose generators
$x_1, \dots, x_n$ form a geometric progression are QNR over $\C$.
\end{itemize}

\ignore{  %
\begin{corollary}
The Paley--Hadamard matrices are QNR over $\C$.
\end{corollary}
This result is interesting because the Paley--Hadamard matrices are an exponentially more frequent family of Hadamard matrices than the Walsh--Hadamard matrices. 
\begin{corollary}
Vandermonde matrices $V_n(x_1, \dots, x_n)$ whose generators
$x_1, \dots, x_n$ form a geometric progression are QNR over any
finite field or $\C$.
\end{corollary}

It has been conjectured that Vandermonde matrices with distinct
generators are rigid.  Our result refutes this conjecture.

For the proofs of these two corollaries, see Section \ref{sec:kivva}.
}  %

\subsection{Overview of the proof}   %
We now take a more detailed look at the techniques used in the proof of
Theorem \ref{allabelian_intro}.

In general, matrices that we deal with will be over $\C$ except in
Sections \ref{sec_finitefield}          %
and \ref{sec_groupalgebra_finitefield}  %
where we extend our results to matrices over finite fields. 
First we define two families of matrices that we will use extensively.
\begin{define}[Generalized Walsh--Hadamard (GWH) matrices] \label{def:GWH1}
The generalized Walsh--Hadamard (GWH) matrix $H_{d,n}$ is a $d^n \times d^n$
complex matrix that has rows and columns indexed by $\Z_d^n$ and entries
$(H_{d,n})_{I,J} = \omega^{I \cdot J}$ where
$\omega = e^{2 \pi i/d}$.  %
\end{define}
Next we define the Discrete Fourier Transform (DFT) matrices. %
\begin{define}[DFT matrix]
The  $(x,y)$ entry of the $N\times N$ matrix $\DFT_N$ \
$(0\le x,y\le N-1)$ 
is $\omega^{xy}$ where $\omega=e^{2\pi i/N}$.
\end{define}
Note $\DFT_N = H_{N,1}$ and 
$H_{d,n} = \underbrace{\DFT_d \otimes \cdots \otimes \DFT_d}_{\text{$n$}}$  
where $\otimes$ denotes the Kronecker product. 

One key idea in our %
argument is the observation that, if all members of a
family $\mathcal A$ of matrices are simultaneously
diagonalizable by a matrix $M$, then
the rigidity of \emph{any} matrix $A \in \mathcal A$ implies the rigidity of
the matrix $M$ (Lemma \ref{diagonalization}).  %
This situation happens, e.g., when $\mathcal A$ is the family of circulant matrices and $M$ is the DFT matrix. 
This simple, yet crucial observation allows us to deduce the non-rigidity of a larger family of matrices.\footnote{The observation that DFT 
matrices diagonalize circulant matrices has similar algorithmic implications as if there were, say, linear-size    
circuits for computing the DFT matrix then we would be able to obtain
linear-size circuits for computing any convolution.}

\subsubsection{Generalized Walsh--Hadamard matrices} 
The first step in the proof of Theorem \ref{allabelian_intro}
is proving that the generalized Walsh--Hadamard matrices are not rigid
in the following sense, which is stronger than QNR.  %
\begin{theorem}[Generalized Walsh--Hadamard matrices are not rigid]  %
\label{Hadamard_informal}
For fixed $d$ and $0 < \epsilon < 0.01$, there exists an $\epsilon'$ such that for all sufficiently large $n$,  $\textsf{r}_{H_{d,n}}\left( d^{n(1 - \epsilon')} \right) \leq d^{n\epsilon}$. \end{theorem}

Note that Theorem \ref{Hadamard_informal} generalizes the main result of \cite{old_hadamard}
(which deals with $d = 2$).   %
The result of \cite{old_hadamard} is stronger in the sense that it holds over $\Q$ while our results for GWH matrices require working over a field extension.  See Section \ref{sec:differentfield} for a
discussion on rigidity over different fields.

Also, given any $d^n \times d^n$ matrix of the form $M_{xy} = f(x - y)$ with  $f:\Z_d^n  \rightarrow \C$, we can permute its rows so that it is diagonalized by $H_{d,n}$.  Thus, we can apply the diagonalization trick mentioned above and obtain the following result, which extends
the results  %
in \cite{CLP} to matrices over $\C$.
\begin{corollary}\label{additive_informal}
Let $f$ be a function from $\Z_d^n \rightarrow \C$ and let $M$ be a $d^n \times d^n$ matrix with $M_{xy} = f(x - y)$.  Then for any fixed $d$ and $0 < \epsilon < 0.01$, there exists an $\epsilon' > 0$ such that for all sufficiently large $n$, %
we have %
$\textsf{r}_{M}\left( d^{n(1 - \epsilon')} \right) \leq d^{n\epsilon}$.
\end{corollary}

\subsubsection{DFT matrices} 

Equipped with the machinery for
Generalized Walsh--Hadamard (GWH) matrices,
the next step is to prove non-rigidity for DFT 
matrices.  The result we prove is the following.
\begin{theorem}[DFT Matrices are Not Rigid]\label{fullFourier_intro}
The family of DFT matrices $\DFT_N$
where $N \in \N$ is QNR over $\C$.
\end{theorem} 
Our proof consists of two steps.  First we show that for integers $N$ of a very special form, the $N \times N$ DFT 
 matrix is not rigid because it can be decomposed into submatrices with 
GWH-type structure.  We say an integer $N$ is
\emph{well-factorable}  %
if it is a product of distinct primes $q_1, \dots , q_l$ such that %
$q_i - 1$ has no large prime power divisors for all $i$.
We will make this notion more precise later, but informally, the first step is
as follows. 
\begin{theorem} \label{main_informal}
Let $\cal A$ denote the family of DFT matrices $\DFT_N$
 where $N$ is well-factorable.  Then the family $\cal A$ is QNR over $\C$.
\end{theorem}
The main intuition is that if $N$ is a product of distinct primes
$q_1, \dots , q_l$, then within the DFT matrix $\DFT_N$,
we can find submatrices whose rows and columns can be indexed by
$\F_{q_1}^{\times} \times \cdots \times \F_{q_l}^{\times}$ where
$\F^{\times}$ denotes the multiplicative group of the field $\F$.
This multiplicative structure can be replaced by the additive structure of $\Z_{q_1-1} \times \cdots \times \Z_{q_l-1}$.  We can then factor each additive group $\Z_{q_i-1}$ into prime power components.  If $q_1-1, \dots , q_l-1$ all have no large prime power divisors, we expect prime powers to be repeated many times when all of the terms are factored. This allows us to find submatrices with $\Z_d^l$ additive structure
to  %
which we can apply tools such as Theorem \ref{Hadamard_informal} and Corollary \ref{additive_informal} to reduce the rank while changing a small number of entries.  We then bound the rank and total number of entries changed over all submatrices to deduce that $\DFT_N$ is not rigid.

The second step of our proof that DFT
matrices are not rigid involves extending Theorem \ref{main_informal} to all values of $N$. The diagonalization trick gives that $N \times N$ circulant matrices are not rigid when $N$ is well-factorable.  We then show that for $N' < \frac{N}{2}$, we can rescale the columns of the $N' \times N'$ DFT 
 matrix and embed it into an $N \times N$ circulant matrix.  As long as $N'$ is not too much smaller than $N$ (say $N' > \frac{N}{(\log N)^2}$), we get that the $N' \times N'$ DFT 
 matrix is not rigid.  Thus, for each well-factorable $N$ and all $N'$ in the range $\frac{N}{(\log N)^2} < N' < \frac{N}{2}$, the $N' \times N'$ DFT 
matrix is not rigid.  We then use a number theoretic result of
Baker and Harman  %
\cite{smooth_primes} to show that the
multiplicative gaps  %
between well-factorable integers are not too large. Thus, the above intervals cover all integers as $N$ runs over all well-factorable numbers, finishing the proof.

As a corollary to
Theorem \ref{fullFourier_intro}  
(due to the diagonalization trick), we get that circulant matrices are not rigid.
\begin{theorem}[Circulant Matrices are not Rigid]\label{circulant_intro}
Let $\cal A$ denote the family of circulant matrices.  Then $\cal A$ is QNR over $\C$.
\end{theorem}

Also notice that since any Toeplitz matrix of size at most $\frac{N}{2}$ can be embedded in an $N \times N$ circulant matrix, the above implies an analogous result for all Toeplitz matrices.  While \cite{toeplitz} shows nontrivial rigidity lower bounds for rank much smaller than $N$, our results imply that there are actually no nontrivial rigidity lower bounds for rank close to $N$.
\subsubsection{$G$-circulant matrices}
We extend our results for DFT and circulant matrices to
$\DFT_G$ and %
$G$-circulant matrices
for finite abelian groups $G$ %
by using the fundamental theorem of finite abelian groups to write $G$ as a direct product of cyclic groups.  Note that any $G$-circulant matrix is diagonalized by the $\DFT_G$ matrix which is %
the Kronecker product of the DFT matrices for the individual cyclic groups.  If there are many small cyclic groups in the product, then we can use the same techniques that we use for GWH-matrices while if there are enough large cyclic groups, then we can rely on our results for DFT matrices.
\subsection{Rigidity over different fields}\label{sec:differentfield}
There are several interesting questions that arise when considering
rigidity over different fields.
Our results for circulant and DFT matrices require working over a
field extension.
The matrix $\DFT_N$ is defined over the $N$-th cyclotomic field
$\Q[\omega]$ where $\omega$ is a primitive $N$\ts{th} root of unity.
But we are not able to show non-rigidity of this matrix over
$\Q[\omega]$, only over a larger field that includes additional
roots of unity.  Therefore the same holds for circulant matrices
over $\Q$ whose non-rigidity we derived from the non-rigidity of
$\DFT_N$.  %
The degree of the extension is
$\tilde{O}(N^2)$ (so combined with $\omega$, the entire extension of $\Q$ has degree $\tilde{O}(N^3)$). See Theorem \ref{thm:extension_degree} for more details.
We leave it as an open question whether our results for fields of characteristic $0$ still hold without field extensions or for extensions of lower degree.

In Section \ref{sec_finitefield} we  
extend our results for complex-valued matrices to any finite field, $\F_q$.  The main idea is to first work over an extension, say $\F_q[\alpha]$, where the matrices are not rigid, and then sum over the matrices obtained by replacing $\alpha$ with its conjugates.  A key ingredient in our proof is that over finite fields, primitive $n$\ts{th}
roots of unity may have minimal polynomial with very low degree, even subpolynomial in $n$.  However, over $\Q$, the primitive $n$\ts{th} roots of unity have minimal polynomial with degree $\phi(n)$ (which is $n^{1-o(1)}$) so our argument does not generalize to this case.  We leave it as an open question whether
circulant   %
matrices are rigid over $\Q$.  It is worth noting that the result of Alman and Williams in \cite{old_hadamard} does hold over $\Q$ while the result of Dvir and Edelman in \cite{CLP} holds only when the field is a finite field related to the group structure of the matrix.

Finally, over a finite field $\F_q$, our result for $G$-circulant matrices requires that $\gcd(q,|G|) = 1$.  Our result for circulant matrices (i.e. cyclic groups) over finite fields does not require this assumption.  The reason that we need $\gcd(q,|G|) = 1$ for general abelian groups is that our techniques do not deal with groups such as $G = \Z_{p^2} \times \dots \times \Z_{p^2}$ (where $p$ is the characteristic of $\F_q$) because $p$\ts{th} roots of unity do not exist over any extension of $\F_q$ so we cannot diagonalize $G$-circulant matrices even if we lift to a field extension.  This is not an issue for large cyclic groups because for a cyclic group, say $\Z_N$, we can embed a $\Z_N$-circulant matrix in a circulant matrix of any given size at least $2N$.  We only require the condition $\gcd(q,|G|) = 1$ to rule out the case when $G$ contains a direct product of many copies of the same small cyclic group whose order is not relatively prime to $q$.  See Section \ref{sec_groupalgebra_finitefield} for more details.  We leave it as an open question whether our results for $G$-circulant matrices still hold without the condition that $\gcd(q,|G|) = 1$.  The results of Dvir and Edelman in \cite{CLP} deal with a special case where $\gcd(q,|G|) > 1$, namely when $G$ is a direct product of many copies of $\Z_p$ where $p$ is the characteristic of $\F_q$.

\subsection{Organization}
In Section \ref{sec_prelim}, we introduce notation and prove several basic results that we will use throughout the paper.  In Section \ref{sec_Hadamard}, we show that %
GWH matrices  %
and several closely related families of matrices are not rigid.  In Section \ref{sec_Fourier1}, we show that $N \times N$ %
DFT %
 matrices are not rigid when $N$ satisfies certain number-theoretic
conditions.   %
In Section \ref{sec_Fourier2}, we complete the proof that %
no (sufficiently large) DFT matrix is rigid%
. We then deduce that %
 Toeplitz matrices are not rigid.  In Section \ref{sec_groupalgebra}, we use the results from the previous section to show that 
$G$-circulant    %
matrices for abelian groups $G$   %
are not rigid.  {}From Section \ref{sec_prelim} through Section \ref{sec_groupalgebra}, we work with matrices over $\C$ for ease of exposition.  In Section \ref{sec_finitefield} and Section \ref{sec_groupalgebra_finitefield}, we sketch how to modify the proofs in the previous sections to deal with ``missing'' roots of unity in a finite field.  
Finally, in Section \ref{sec_conclusion}, we discuss a few open questions
and possible directions for future work.

\section{Preliminaries}\label{sec_prelim}
Throughout this paper, we let $d \geq 2$ be an integer and $\omega = e^{2\pi i/d}$ be a primitive $d$\ts{th} root of unity.  When we consider an element of $\Z_d^n$, we will view it as an ordered $n$-tuple with entries in the range $[0,d-1]$.  When we say a list of $d^n$ elements $x_1, \dots , x_{d^n}$ is indexed by $\Z_d^n$, we mean that each $x_i$ is labeled with an element of $\Z_d^n$ such that all labels are distinct and the labels of $x_1, \dots , x_{d^n}$ are in lexicographical order. 

\subsection{Basic notation}     %
We will frequently work with ordered tuples, say $I = (i_1, \dots , i_n) \in \Z_d^n$.  Below we introduce some notation for dealing with ordered tuples that will be used later on.

\begin{define}
For an ordered tuple $I$, %
we let $I^{(i)}$ denote its $i$\ts{th} entry.  
For instance if $I = (i_1, \dots, i_n)$ then $I^{(k)} = i_k$.
\end{define}

\begin{define}\label{def:tupleexponent}
For an ordered $n$-tuple $I = (i_1,i_2, \dots ,i_n)$, define the polynomial
in    %
$n$ variables $x^I = x_1^{i_1} \cdots x_n^{i_n}$.
\end{define}

\begin{define}
For $\omega$ a $d$\ts{th} root of unity and an ordered $n$-tuple $I = (i_1,i_2, \dots ,i_n) \in \Z_d^n$, we define $\omega^{[I]} = (\omega^{i_1}, \dots , \omega^{i_n})$.

\end{define}

\begin{define}
For a function $f: \Z_d^n \rightarrow \C$, define the $n$-variable polynomial $P_f$ as 
\[
P_f = \sum_{I \in \Z_d^n} f(I)x^I \,.
\]
\end{define}

\begin{define}
For an ordered $n$-tuple $I = (i_1,i_2, \dots ,i_n)$, we define the set $\perm(I)$ to be a set of ordered $n$-tuples consisting of all distinct permutations of the entries of $I$.  Similarly, for a set of ordered $n$-tuples $S$, we define $\perm(S)$ to be the set of all ordered $n$-tuples that can be obtained by permuting the entries of some element of $S$.
\end{define}

\begin{define}
We say a set $S \subseteq \Z_d^n$ is symmetric if %
$\perm(I) \subseteq S$ for any $I \in S$. %
\end{define}

\begin{define}
For a set of ordered $n$-tuples $S$, let $\red(S)$ denote the set of equivalence classes under permutation of entries in $S$.  Let $\rep(S)$ be a set of ordered $n$-tuples formed by taking one representative from each equivalence class in $\red(S)$ (note $\rep(S)$ is not uniquely determined but this will not matter for our purposes).  
\end{define}
Note that if $\rep(S) = \{I_1, \dots , I_k \}$, then the sets $\perm(I_1), \perm(I_2), \dots , \perm(I_k)$ are disjoint and their union contains $S$.  If the set $S$ is symmetric then their union is exactly $S$.

\subsection{Special families of matrices} \label{sec:specialmatrices}   %
We now define notation for working with a few special families of matrices.
\begin{define}
An $N \times N$ matrix $M$ is called a Toeplitz matrix if $M_{ij}$ depends only on $i-j$.  An $N \times N$ matrix $M$ is called a Hankel matrix if $M_{ij}$ depends only on $i + j$.  Note that the rows of any Toeplitz matrix can be permuted to obtain a Hankel matrix so any non-rigidity results we show for one family also hold for the other.
\end{define}

\begin{define}[Adjusted $G$-circulant matrices] %
For an abelian group $G$ and a function $f: G \rightarrow \C$, let $M_G(f)$ denote the $|G| \times |G|$ matrix (over $\C$) whose rows and columns are indexed by elements $x,y \in G$ and whose entries are given by $M_{xy} = f(x + y)$.  When it is clear what $G$ is from context, we will simply write $M(f)$.  We let $V_G$ denote the family of matrices $M_G(f)$ as $f$ ranges over all functions from $G$ to $\C$.  We call $V_G$ the family of 
\emph{adjusted   %
$G$-circulant    %
matrices} for the group $G$.  When $G$ is a cyclic group, we call the matrices in $V_G$ adjusted-circulant. 
\end{define}
Compared to the usual 
$G$-circulant    %
(and circulant) matrices defined by $M_{xy} =  f(x-y)$, the matrix $M_G(f)$ differs only in a permutation of the rows.  In the
subsequent  %
sections, we will work with $M_G(f)$ for technical reasons, but it is clear that the same non-rigidity results hold for the usual 
$G$-circulant    %
matrices.  Similarly, we will use adjusted-circulant and Hankel matrices as it is clear that the same non-rigidity results hold for circulant and Toeplitz matrices.  Also note that adjusted-circulant matrices are a special case of Hankel matrices.

Recall that a \emph{character} of an abelian group $G$ is a
homomorphism from $G$ to $\C^{\times}$, the multiplicative group
of complex numbers.   %
\begin{define}[Discrete Fourier Transfrom matrices] %
For a finite abelian group $G$, we
define $\DFT_G$, the DFT matrix for $G$, %
as the $|G| \times |G|$ matrix whose rows correspond to elements of $G$ and
whose columns correspond to the characters of the group. %
To simplify notation, we will write 
$\DFT_N$ for %
$\DFT_{\Z_N}$, the classical $N\times N$ Fourier Transform matrix  %
(for the cyclic group $\Z_N$).  %
\end{define}
The following is immediate from the definition. %
\begin{fact.}\label{fact:dftproduct}
For a finite abelian group $G$, if $G = H \times K$ where $H$ and $K$ are
subgroups,   %
there is an ordering of the rows and columns of $G$ so that $\DFT_G = \DFT_H \otimes \DFT_K$.  In particular, if $G = \Z_{n_1} \times \Z_{n_2} \times \cdots \times \Z_{n_a}$ then 
$$\DFT_G = \DFT_{n_1} \otimes \cdots \otimes \DFT_{n_a}\,.$$
\end{fact.}
\subsection{Matrix rigidity}    %
Here, we review basic notation for matrix rigidity.
\begin{define}
For a matrix $M$ and a real number $r$, we define $\textsf{R}_M(r)$ to be the smallest number $s$ for which there exists a matrix $A$ with at most $s$ nonzero entries and a matrix $B$ of rank at most $r$ such that $M = A+B$.  If $\textsf{R}_M(r) \geq s$, we say $M$ is $(r,s)$-rigid.
\end{define}

\begin{define}
For a matrix $M$ and a real number $r$, we define $\textsf{r}_M(r)$ to be the smallest number $s$ for which there exists a matrix $A$ with at most $s$ nonzero entries in each row and column and a matrix $B$ of rank at most $r$ such that $M = A+B$.  If $\textsf{r}_M(r) \geq s$, we say $M$ is $(r,s)$-regular rigid.
\end{define}
It is clear that if a matrix is $(r,ns)$-rigid, then it must be $(r,s)$-regular rigid.  In
the following  %
sections, we will show that various matrices are not $(\frac{N}{\log \log N}, N^{\epsilon})$-regular rigid for any $\epsilon >0$ and this will imply that Valiant's method for showing circuit lower bounds in \cite{rigidity} cannot be applied for these matrices.

\subsection{Preliminary results}    %
Next, we mention several basic results that will be useful in the proofs later on.
\begin{define}
For an $m \times n$ matrix $A$ and $p \times q$ matrix $B$, the Kronecker product $A \otimes B$ is the $mp \times nq$ matrix given by
\[
\begin{bmatrix}
a_{11}B & \dots & a_{1n}B \\
\vdots & \ddots & \vdots \\
a_{m1}B & \dots & a_{mn}B
\end{bmatrix}
\]
where $a_{ij}$ are the entries of $A$.

\end{define}

\begin{fact.}    \label{product-kronecker}
For matrices $A,B,C,D$ such that matrix products $AC$ and $BD$ are defined,
\[
(A \otimes B)(C \otimes D) = (AC) \otimes (BD) \,.
\]
\end{fact.}

\begin{claim}
$H_{d,n} = \underbrace{\DFT_d \otimes \cdots \otimes \DFT_d}_{\text{$n$}}$
where $\otimes$ denotes the Kronecker product.
\end{claim}
\begin{proof}
This can easily be verified from the definition.
\end{proof}
\begin{claim}\label{orthogonal}
$H_{d,n}H_{d,n}^* = d^nI$ where $H_{d,n}^*$ is the conjugate transpose of $H_{d,n}$ and $I$ is the identity matrix.
\end{claim}
\begin{proof}
We verify that %
$\DFT_d\DFT_d^* = dI$, %
and then
use Claim \ref{orthogonal} and
Fact \ref{product-kronecker}.     %
\end{proof}

\begin{claim} \label{Hadamard_diagonalization}
Let $f: \Z_d^n \rightarrow \C$ be a function.  Let $\omega$ be a %
primitive  %
$d$\ts{th}
root of unity and set $P_f = \sum_{I \in \Z_d^n} f(I)x^I$ %
(see Definition \ref{def:tupleexponent})%
.  Let $D = H_{d,n}M_{\Z_d^n}(f)H_{d,n}$.  Then $D$ is a diagonal matrix with diagonal entries $d^nP_f(\omega^{[J]})$ as $J$ ranges over $\Z_d^n$.
\end{claim}
\begin{proof}
First, we analyze the product $M_{\Z_d^n}(f)H_{d,n}$.  This is a $d^n \times d^n$ matrix and its rows and columns can naturally be indexed by ordered tuples $I,J \in \Z_d^n$.  The entry with row indexed by $I$ and column indexed by $J$ is
\[
\sum_{I' \in \Z_d^n} f(I + I')\omega^{I' \cdot J} = \omega^{-I \cdot J}\sum_{I' \in \Z_d^n} f(I + I')\omega^{(I'+I) \cdot J} = \omega^{-I \cdot J} P_f(\omega^{[J]}) \,.
\]
Therefore, the columns of $M_{\Z_d^n}(f)H_{d,n}$ are multiples of the columns of $H_{d,n}^*$.  In fact, the column of $M_{\Z_d^n}(f)H_{d,n}$ indexed by $J$ is $P_f(\omega^{[J]})$ times the corresponding column of $H_{d,n}^*$.  Since $H_{d,n}H_{d,n}^* = d^nI$, %
we deduce that %
$D$ must be a diagonal matrix whose entries on the diagonal are $d^nP_f(\omega^{[J]})$ as $J$ ranges over $\Z_d^n$. 
\end{proof}

\begin{claim}\label{Fourier_diagonalization}
Let $M$ be a $d \times d$ adjusted-circulant matrix.  Then %
$\DFT_d \cdot  M  \cdot \DFT_d$ %
is a diagonal matrix.
\end{claim}
\begin{proof}
Plug $n=1$ into the above.
\end{proof}

Claim \ref{Hadamard_diagonalization} gives us a characterization of the rank of matrices of the form $M_{\Z_d^n}(f)$.
\begin{claim} \label{rankcomp}
Let $f: \Z_d^n \rightarrow \C$ be a function.  Let $\omega$ be a %
$d$\ts{th}
root of unity and 
assume %
$P_f = \sum_{I \in \Z_d^n} f(I)x^I$ has $C$ roots among the set %
$\{ (\omega^{i_1}, \dots , \omega^{i_n}) \mid (i_1, \dots , i_n) \in \Z_d^n \}$ %
.
Then $\rank(M_{\Z_d^n}(f)) = d^n - C$.
\end{claim}
\begin{proof}
Consider the product $D = H_{d,n}M_{\Z_d^n}(f)H_{d,n}$.  Note that $H_{d,n}$ is clearly invertible by Claim \ref{orthogonal}.  Therefore, it suffices to compute the rank of $D$.  By Claim \ref{Hadamard_diagonalization}, $D$ must be a diagonal matrix whose entries on the diagonal are $d^nP_f(\omega^{[J]})$ as $J$ ranges over $\Z_d^n$.  The rank of $D$ is the number of nonzero diagonal entries which is simply $d^n - C$.
\end{proof}

As mentioned in the introduction, we can relate the rigidity of a matrix to the rigidity of matrices  that it diagonalizes.
\begin{lemma}\label{diagonalization}
If $B = A^*DA$ where $D$ is a diagonal matrix and $\textsf{r}_A(r) \leq s$ then $\textsf{r}_B(2r) \leq s^2$.  The same inequality holds also for $B' = ADA$.
\end{lemma}
\begin{proof}
Let $E$ be the matrix with at most $s$ nonzero entries in each row and column such that $A - E$ has rank at most $r$.  We have
\[
B - E^*DE = A^*D(A - E) + (A^* - E^*)DE \,.
\]
Since $\rank(A - E) \leq r$, %
we get that %
$\rank(B - E^*DE) \leq 2r$.  Also, $E^*DE$ has at most $s^2$ nonzero entries in each row and column so $\textsf{r}_B(2r) \leq s^2$.  The second part can be proved in the exact same way with $A^*$ replaced by $A$.
\end{proof}  
In light of Lemma \ref{diagonalization}, Claim \ref{Fourier_diagonalization}, and Claim \ref{Hadamard_diagonalization}, proving non-rigidity for $d \times d$ circulant matrices reduces to proving non-rigidty for %
$\DFT_d$ %
 and proving non-rigidity for 
$G$-circulant    %
matrices for
$G=\Z_d^n$  %
reduces to proving non-rigidity for $H_{d,n}$.  Below, we show that these statements are actually equivalent.

\begin{claim}\label{rescaling}
It is possible to rescale the rows and columns of $H_{d,n}$ to get a matrix of the form $M_{\Z_d^n}(f)$ for some symmetric function $f:\Z_d^n \rightarrow \C$.  In particular, it is possible to rescale the rows and columns of %
$\DFT_d$ %
 to get an adjusted-circulant matrix.
\end{claim}
\begin{proof}
Let $\zeta$ be such that $\zeta^2 = \omega$.  Multiply each row of $H_{d,n}$ by $\zeta^{(I \cdot I)}$ and each column by $\zeta^{(J \cdot J)}$ to get a matrix $H'$.  We have
\[
H'_{IJ} = \zeta^{(I+J) \cdot (I+J)} \,.
\]
For an ordered tuple $x = (x_1, \dots , x_n) \in \Z_d^n$, we define $f(x) = \zeta^{x_1^2 + \dots + x_n^2}$.  To complete the proof, it suffices to show that $f: \Z_d^n \rightarrow \C$ is well defined.  To do this, we will show that $\zeta^{x^2}$ depends only on the residue of $x \mod d$.  If $d$ is odd, we can choose $\zeta$ to be a $d$\ts{th} root of unity and the claim is clear.  If $d$ is even $\zeta^{(x+d)^2} = \zeta^{x^2}\zeta^{2dx + d^2}$ but since $2dx + d^2$ is a multiple of $2d$, %
we get that  %
$\zeta^{2dx + d^2} = 1$ and thus $\zeta^{(x+d)^2} = \zeta^{x^2}$.  
\end{proof}

\section{Non-rigidity of generalized Walsh--Hadamard matrices}
\label{sec_Hadamard}
In this section, we show that the %
GWH  %
matrix $H_{d,n}$ becomes highly non-rigid for large values of $n$.  The precise result is stated below.
\begin{theorem} \label{Hadamard}
Let $N = d^n$ for positive integers $d,n$.   Let $0 < \epsilon < 0.01 $ and assume $n \geq 1/\psi$ where \[\psi = \frac{\epsilon^2 }{400\log^2 (1/\epsilon) d \log d} \,.\]
Then
\[
\textsf{r}_{H_{d,n}}\left(N^{1 - \psi}\right) \leq N^{ \epsilon} \,.
\]
\end{theorem}

First we prove a few lemmas about symmetric polynomials that we will use in the proof of Theorem \ref{Hadamard}.
\begin{lemma}\label{symmetric}
  Let $T_m$ denote the set of ordered tuples in $\Z_d^n$ such that at least $m$ entries are equal to $0$.
Let   %
$\rep(T_m) = \{I_1, \dots , I_k \}$.  Consider the  polynomials $P_1(x_1, \dots , x_n), \dots , P_k(x_1, \dots , x_n)$ defined by
\[
P_i(x_1, \dots , x_n) = \sum_{I \in \perm(I_i)} x^{I} \,.
\]
For any complex numbers  $y_1, \dots , y_m$, and any polynomial $Q(x_{m+1}, \dots x_n)$ that is symmetric and degree at most $d-1$ in each of its variables, there exist coefficients $c_1, \dots , c_k$ such that 
\[
Q(x_{m+1}, \dots , x_n) = \sum c_iP_i(y_1, \dots y_m, x_{m+1}, \dots , x_n) \,.
\]
\end{lemma}
\begin{proof}
It suffices to prove the statement for all $Q$ of the form 
\[
\sum_{I'' \in \perm(I')} x^{I''}
\]
where $I' \in \Z_d^{n-m}$.  We will prove this by induction on the degree.  Clearly one of the $I_i$ is $(0,0 \dots 0)$, so one of the polynomials $P_i(x_1, \dots , x_n)$ is constant.  This finishes the case when $Q$ has degree $0$.  Now we do the induction step.  Note that we can extend $I'$ to an element of $T_m$ by setting the first $m$ entries equal to $0$.  Call this extension $I$ and say that $I \in \perm(I_i)$.  We have
\[
\sum_{I'' \in \perm(I')} x^{I''} = P_i(y_1, \dots ,y_m, x_{m+1}, \dots , x_n) - R(y_1, \dots , y_m, x_{m+1}, \dots x_n) \,.
\]
$R(y_1, \dots , y_m, x_{m+1}, \dots x_n)$, when viewed as a polynomial in $x_{m+1}, \dots , x_n$ (since $y_1, \dots , y_m$ are complex numbers that we can plug in), is symmetric and of lower degree than the left hand side.  Thus, using the induction hypothesis, we can write $R$ in the desired form.  This completes the induction step.  
\end{proof}

The key ingredient in the proof of Theorem \ref{Hadamard} is the following lemma which closely resembles the main result in \cite{CLP}, but deals with matrices over $\C$.
\begin{lemma}\label{non-rigidity1}
Let $f: \Z_d^n \rightarrow \C$ be a symmetric function on the $n$ variables.  Let $N = d^n$.  Let $0 < \epsilon < 0.01 $ and assume $n \geq 1/\psi$ where \[\psi = \frac{\epsilon^2 }{400\log^2 (1/\epsilon) d \log d} \,. \]  Then 
\[
\textsf{r}_{M(f)}\left(N^{1 - \psi}\right) \leq N^{ \epsilon} \,.
\]

\end{lemma}
Let %
\[
\delta = \frac{\epsilon}{10\log (1/\epsilon)},
   \text{\quad and \quad}
m = \left\lceil n\big( \frac{1 - \delta}{d} \big) \right\rceil
\]
and let $S$ denote the set of all ordered tuples $(i_1, i_2, \dots ,i_n) \in \Z_d^n$ such that the entries indexed
$1,2, \dots ,m$   %
are equal to $0$, the entries indexed $m+1, \dots ,2m$ are equal to $1$ and in general for $0 \leq i \leq d-1$, the entries indexed $im + 1, \dots , (i+1)m$ are equal to $i$.  Note $|S| = d^{n-dm} \approx d^{\delta n} = N^{\epsilon^2}$ (since $n - dm$ is approximately $\delta n$).  

The main idea will be to change $f$ in a small number of locations so that it has many zeros in the set %
$\{\omega^{[I]} \mid I \in \Z_d^n \}$ %
in order to make use of Claim \ref{rankcomp}.  More precisely, first we will change $f$ to $f'$ by changing its values in at most $N^{\epsilon}$ places so that $f'$ is still symmetric in all of the variables and
\begin{equation*} \forall I\in S,\qquad P_{f'}\left( \omega^{[I]} \right) = 0 \,.
\end{equation*}   %
Note that although the size of $S$ is small, the fact that $f'$ is symmetric implies that $f'$ also vanishes on $\perm(S)$, which covers almost all of $\Z_d^n$.  Once we have shown the above, we quantitatively bound the number of entries changed between $M(f)$ and $M(f')$ and also the rank of $M(f')$ to complete the proof of Lemma \ref{non-rigidity1}.  To do the first part, we need the following sub-lemma.
\begin{lemma}\label{changes}
Let $T$ denote the set of all ordered tuples $(i_1, i_2, \dots ,i_n) \in \Z_d^n$  such that at least $n\big(1 - \delta \big)$ of the entries are $0$. By changing the values of $f$ only on elements of $T$, we can obtain $f'$ satisfying 
\begin{equation}\label{roots} \forall I\in S, \qquad
P_{f'}\left( \omega^{[I]} \right) = 0 \,.
\end{equation}
\end{lemma}

\begin{proof}

We interpret (\ref{roots}) as a system of linear equations where the unknowns are the values of $f'$ at various points.  Let $\rep(T) = \{J_1, J_2, \dots , J_k \}$ for $J_1, J_2, \dots J_k \in T$.  Since we must maintain that $f'$ is symmetric, there are essentially $k$ variables each corresponding to an equivalence class of ordered tuples under permutations.  Each equivalence class is of the form $\perm(J_j)$ and we denote the corresponding variable by $m_j$.  The system of equations in (\ref{roots}) can be rewritten in the form
\[    %
\forall I \in S\qquad 
\sum_{j=1}^k m_j \sum_{J \in \perm(J_j)} \omega^{I \cdot J} + \sum_{J' \notin T}f(J') \omega^{I \cdot J'} = 0 \,.
\]
If we let  $\rep(S) = \{ I_1, I_2, \dots , I_l\}$, the system has exactly $l$ distinct equations corresponding to each element of $\rep(S)$ due to our symmetry assumptions.  Let $M$ denote the $l \times k$ coefficient matrix represented by $M_{ij} = \sum_{J \in \perm(J_j)} \omega^{I_i \cdot J}$.  To show that the system has a solution, it suffices to show that the column span of $M$ is full.  This is equivalent to showing that for each $i = 1,2, \dots, l$ there exist coefficients $a_1,a_2, \dots ,a_k$ such that
\begin{align*}
& \sum_{j=1}^{k} a_j \cdot\sum_{J \in \perm(J_j)} \omega^{I_i \cdot J}  \neq 0\,,  \\
\forall i' \neq i \qquad
& \sum_{j=1}^{k} a_j \cdot\sum_{J \in \perm(J_j)} \omega^{I_{i'} \cdot J} =0  \,.
\end{align*}

Fix an index $i_0$.  We can view each equation above as a polynomial in $\omega^{[I_{i}]}$ given by 
\[
P(x_1, \dots ,x_n) = \sum_{j=1}^k a_j \sum_{J \in \perm(J_j)} x^J
\]
and the problem becomes equivalent to constructing a polynomial that vanishes on $\omega^{[I_{i}]}$ if and only if $i \neq i_0$.  Note that only the entries $x_{dm+1}, \dots , x_n$ matter as we have \[x_1 = \dots = x_m = 1, \dots , x_{(d-1)m+1} = \dots = x_{dm} = \omega^{d-1}\] for all points we consider.

For $I_i = (i_1, i_2, \dots i_n)$, let $I'_i$ denote the (ordered) sub-tuple $(i_{dm+1}, \dots , i_n)$.  The problem is equivalent to constructing a polynomial  
\[
Q(x_{dm+1}, \dots ,x_n) = P(1,1, \dots, \omega^{d-1}, \dots, \omega^{d-1}, x_{dm+1}, \dots x_n)
\]
such that $Q$ vanishes on $\omega^{[I'_{i}]}$ if and only if $i \neq i_0$.

Lemma \ref{symmetric} implies that by choosing the coefficients $a_1, \dots , a_k$, we can make $Q$ be any polynomial that is symmetric in $x_{dm+1} ,\dots ,x_n$ and degree at most $d-1$ in each of the variables.

Now consider the polynomial 
\[
Q_{i_0}(x_{dm+1}, \dots , x_n) = \sum_{I' \in \perm(I'_{i_0})} \bigg( \frac{x_{dm+1}^d-1}{x_{dm+1}-\omega^{I'^{(1)}}} \bigg) \cdots \bigg( \frac{x_n^d-1}{x_n-\omega^{I'^{(n-dm)}}} \bigg) \,.
\]
Note this is a polynomial with coefficients in $\C$ since each of the factors reduces to a degree $d-1$ polynomial.

It is clear that the above polynomial is symmetric in all of the variables and satisfies the degree constraint so we know we can choose suitable coefficients $a_1, \dots  , a_k$.  We claim that the polynomial we construct does not vanish on $\omega^{[I'_{i_0}]}$ but vanishes on $\omega^{[I'_{i}]}$ for $i \neq i_0$.  Indeed, the product
\[
\bigg( \frac{x_{dm+1}^d-1}{x_{dm+1}-\omega^{I'^{(1)}}} \bigg) \cdots \bigg( \frac{x_n^d-1}{x_n-\omega^{I'^{(n-dm)}}} \bigg)
\]
is $0$ if and only if $(x_{dm+1}, \dots , x_{n}) \neq I'$.  However, there is exactly one $I' \in \perm(I'_{i_0})$ with $I' = I'_{i_0}$ and none with $I' = I'_{i}$ for $i \neq i_0$ since $I_1, I_2, \dots , I_l$ are representatives of distinct equivalence classes under permutation of entries.  This means that the polynomial $Q_{i_0}$ we constructed has the desired properties and completes the proof that the system is solvable.
\end{proof}

\begin{proof}[Proof of Lemma \ref{non-rigidity1}] Since $M(f) = (M(f) - M(f')) + M(f')$, to complete the proof of Lemma \ref{non-rigidity1}, it suffices to bound the number of nonzero entries in $M(f) - M(f')$ and the rank of $M(f')$.

The number of nonzero entries in each row and column of $(M(f) - M(f'))$ is at most $|T|$.  This is exactly the number of elements of $\Z_d^n$ with at least $n\big(1- \delta \big)$ entries equal to $0$.  Using standard tail bounds on the binomial distribution (see \cite{binomial_tails}), the probability of a random ordered $n$-tuple having at least that many $0$s is at most
\begin{align*}
\exp\left(-nD\left(1-\delta \mid \mid \frac{1}{d}\right)\right) = \exp\left(-n \left( (1- \delta) \log (d(1-\delta)) + \delta \log\left(\frac{d\delta}{d-1}\right)\right)\right) \\ 
= d^{-n(1-\delta)}\exp\left(-n\left((1-\delta) \log(1-\delta) + \delta \log\left(\frac{d\delta}{d-1}\right)\right)\right)
\end{align*}
where %
\[
D(a \mid \mid b) = a \log \frac{a}{b} + (1-a) \log \frac{1-a}{1-b} 
\]
denotes the KL-divergence between Bernoulli distributions with means $a$ and $b$.

 For $\delta < 0.01$, the above is at most $d^{-n(1- 4\delta \log (1/\delta))}$.  Since $4\delta \log (1/\delta) < \epsilon$, we change at most $d^{ \epsilon n}$ entries in each row and column.

By Claim \ref{rankcomp}, the rank of $M(f')$ is at most $d^n - |\perm(S)|$.  Equivalently, this is the number of ordered $n$-tuples such that some element in $\{0,1, \dots , d-1 \}$ appears less than $ \frac{(1-\delta)n}{d}$ times.  We use the multiplicative Chernoff bound and then union bound over the $d$ possibilities to get the probability that a randomly chosen ordered $n$-tuple in $\Z_d^n$ is outside $\perm(S)$ is at most
\[
d\exp\left(-\frac{\delta^2n}{2d}\right)  = \exp\left(-\frac{\delta^2n}{2d} + \log d\right) \,.
\]

When $n > \frac{4d(\log d)}{\delta^2}$, the above is at most $d^{-(\delta^2 n)/(4 d \log d)}$ and thus the rank of $M(f')$ is at most  
$d^{(1- \psi)n}$
where 
\[
\psi = \frac{\epsilon^2 }{400 \log^2 (1/\epsilon) d \log d} \,,
\]
completing the proof of Lemma \ref{non-rigidity1}.
\end{proof}

\begin{proof}[Proof of Theorem \ref{Hadamard}]
Combine Claim \ref{rescaling} and Lemma \ref{non-rigidity1}. %
\end{proof}

Using Theorem \ref{Hadamard}, Lemma \ref{diagonalization}, and {Claim} \ref{Hadamard_diagonalization}, we get the following result which extends {Lemma} \ref{non-rigidity1} to matrices where $f$ is not symmetric.
\begin{corollary}\label{groupalgebra1}
For any function $f: \Z_d^n \rightarrow \C$ and any $0 < \epsilon < 0.01$ such that $n \geq 1/ \psi$ where 
\[\psi = \frac{\epsilon^2 }{400\log^2 (1/\epsilon) d \log d} \,,\]
 we have 
\[
\textsf{r}_{M(f)}\left(2N^{1 - \psi}\right) \leq N^{ 2\epsilon}
\]
where $N = d^n$.
\end{corollary}

\section{Non-rigidity of  %
DFT matrices of well-factorable size} %
\label{sec_Fourier1}
Our goal in this section is to show that we can find infinitely many values of $N$ for which the %
DFT %
 matrix %
$\DFT_N$ %
 is highly non-rigid.  The integers $N$ we analyze will be products of many distinct primes $q_i$ with the property that $q_i - 1$ is
smooth (has all prime factors small).  For these values of $N$, we can decompose the matrix %
$\DFT_N$ %
 into several submatrices that are closely related to Hadamard matrices.  We then apply the results from the previous section to show that each submatrix is non-rigid and aggregate over the submatrices to conclude that %
$\DFT_N$ %
 is non-rigid.

We first show precisely how to construct $N$.  We rely on the following number theoretic result, found in \cite{smooth_primes}, that allows us to find a large set of primes $q_i$ for which $q_i - 1$ is %
smooth. 
\begin{define}
For a positive integer $m$, let $\rho^+(m)$ denote the largest prime factor of $m$.  For a fixed positive integer $a$, let
$$
\pi_a(x,y) = |\{p \mid a<p \leq x, \rho^+(p-a) \leq y\} |
$$
where $p$ ranges over all primes.  In other words, $\pi_a(x,y)$ is the number of primes at most $x$ such that $p-a$ is $y$-smooth.
\end{define}
\begin{theorem}[\hspace{1sp}\cite{smooth_primes}]\label{analytic_NT}
There exist constants $x_0,C$ such that for $\beta = 0.2961, x > x_0$ and $y \geq x^{\beta}$ we have~\footnote{\cite{smooth_primes} proves the same inequality with $\pi_a(x,y)$ for any integer $a$ where $x_0$ may depend on $a$ and $C$ is an absolute constant.}
$$
\pi_1(x,y) > \frac{x}{(\log x)^C}\,.
$$

\end{theorem}

Throughout the remainder of this paper, set $C_0 = C+1$ where $C$ is the constant in {Theorem} \ref{analytic_NT}.  The properties that we want $N$ to have are stated in the following two definitions.
\begin{define}
We say a prime $q$ is $(\alpha,x)$-good if the following
conditions hold.    %
\begin{itemize}
    \item $\frac{x}{(\log x)^{C_0}} \leq q \leq x $.
    \item All prime powers dividing $q-1$ are at most $x^{\alpha} $.
\end{itemize}
\end{define}

\begin{define}
We say an integer $N$ is $(l, \alpha, x)$-factorable if the following
conditions hold.    %
\begin{itemize}
    \item $N = q_1 \cdots q_l$ where $q_1, \dots , q_l$ are distinct primes.
    \item $q_1, \dots , q_l$ are all $(\alpha,x)$-good.
\end{itemize}
\end{define}

To show the existence of $(l, \alpha, x)$-factorable integers, it suffices to show that there are many $(\alpha,x)$-good primes.  This is captured in the following lemma.
\begin{lemma}\label{construction}
For a fixed constant $C_0$, any parameter $\alpha > 0.2961$, and sufficiently large $x$ (possibly depending on $\alpha$), there are at least $10x/(\log x)^{C_0}$ distinct $(\alpha,x)$-good primes.
\end{lemma}

\begin{proof}[Proof of {Lemma} \ref{construction}]  Let $y = x^{\beta}$ where $\beta =  0.2961$.  By {Theorem} \ref{analytic_NT}, for sufficiently large $x$, we can find at least
\[ \left\lceil \frac{x}{(\log x)^C} - \frac{x}{(\log x)^{C_0 }} \right\rceil \] primes $p_1, \dots , p_l$ between $x/(\log x)^{C_0 }$ and $x$ such that all prime factors of $p_i - 1$ are at most $x^{\beta}$.  Eliminate all of the $p_i$ such that one of the prime powers in the prime factorization of $p_i - 1$ is more than $x^{\alpha}$.  Note that there are at most $x^{\beta} \log x$ integers in the range $[x^{\alpha}, x]$  that are powers of primes smaller than $x^{\beta}$.  Each of these prime powers can divide at most $x^{1 - \alpha}$ of the elements $\{ p_1 -1 , \dots , p_l-1\}$, so in total, we eliminate at most $x^{1 - \alpha + \beta} \log x$ of the $p_i$.  Thus, for sufficiently large $x$, the number of $(\alpha ,x)$-good primes is at least
\[
\frac{x}{(\log x)^C} - \frac{x}{(\log x)^{C_0 }} - x^{1 - \alpha + \beta} \log x \geq \frac{x}{2(\log x)^C} \,. \qedhere
\]
\end{proof}

For simplicity, we will set $\alpha = 0.3$ by default. 
\begin{define}
We say a prime is $x$-good %
if it is $(0.3,x)$-good.  %
We say an integer $N$ is %
 $(l,x)$-factorable if it is $(l, 0.3, x)$-factorable.
\end{define}  

{Lemma} \ref{construction} implies that for all sufficiently large $x$ and $l \leq \frac{x}{(\log x)^{C_0}}$ (where $C_0$ is an absolute constant), we can find $(l,x)$-factorable integers.  We now show that if we choose $x$ sufficiently large and $N$ to be $(l,x)$-factorable for some 
\[
\frac{x}{(\log x)^{C_0 + 100}} \leq l \leq\frac{x}{(\log x)^{C_0 + 10}} \,, \]
 then %
$\DFT_N$ %
 is highly non-rigid.

\begin{theorem} \label{main}
Let $0 <\epsilon < 0.01$ be some constant.  For $x$ sufficiently large and $N$ a $(l, x)$-factorable number with \[\frac{x}{(\log x)^{C_0 + 100}} \leq l \leq\frac{x}{(\log x)^{C_0 + 10}}\,, \] we must have
\[
\textsf{r}_{\DFT_N}\left( \frac{N}{\exp\left(\epsilon^6 (\log N)^{0.36}\right)} \right) \leq N^{ 7 \epsilon} \,.
\]
\end{theorem}

In order to prove {Theorem} \ref{main}, we will first prove a series of preliminary results that characterize the structure of %
DFT %
 and 
GWH matrices.   %

\subsection{Structure of generalized Walsh--Hadamard and %
DFT %
 matrices}

\begin{lemma}\label{decomp}
Let $n = x_1x_2\cdots x_j$ for pairwise relatively prime positive integers $x_1, \dots , x_j$. There exists a permutation of the rows and columns of
$\DFT_n$, say $\DFT'$, such that 
\[
\DFT' = \DFT_{x_1} \otimes \cdots \otimes \DFT_{x_j}\,,
\]
where $\otimes$ denotes the Kronecker product.
\end{lemma}
\begin{proof}
This follows from {Fact} \ref{fact:dftproduct}.
\end{proof}

\begin{lemma}\label{products}
Let $M = A \otimes B$ where $A$ is an $m \times m$ matrix and $B$ is an $n \times n$ matrix.  For any two integers $r_1,r_2$ we have
\[\textsf{r}_M(r_1n + r_2m) \leq \textsf{r}_A(r_1)\textsf{r}_B(r_2) \,. \]
\end{lemma}
\begin{proof}
The proof of this lemma is similar to the proof of {Lemma} \ref{diagonalization}.  There are matrices $E,F$ with at most $\textsf{r}_A(r_1)$ and $\textsf{r}_B(r_2)$ nonzero entries respectively such that $\rank(A+E) \leq r_1$ and $\rank(B+F) \leq r_2$.  We will now show that $\rank(M - E \otimes F) \leq r_1n + r_2m$.  Indeed
\[
M - E \otimes F = (A+E) \otimes B - E \otimes (B+F)
\]
and the right hand side of the above has rank at most $r_1n + r_2m$ since rank multiplies under the Kronecker product.  Clearly $E \otimes F$ has at most $\textsf{r}_A(r_1)\textsf{r}_B(r_2)$ nonzero entries in each row and column so we are done.
\end{proof}

\begin{lemma}\label{productbound}
Consider the matrix 
\[
A = (\underbrace{\DFT_{t_1} \otimes \cdots \otimes \DFT_{t_1}}_\text{$a_1$}) \otimes \cdots \otimes  (\underbrace{\DFT_{t_n} \otimes \cdots \otimes \DFT_{t_n}}_\text{$a_n$})\,. 
\]
Let $0 < \epsilon < 0.01$ be some chosen parameter and $D$ be some sufficiently large constant (possibly depending on $\epsilon$).  Assume 
$t_1 \leq t_2 \leq \dots \leq t_n$ 
and $a_i \geq \max\left(\frac{t_i^2(\log t_i)^2}{\epsilon^{10}}, D \right)$ for all $i$.  Let $P = t_1^{a_1} \cdots t_n^{a_n}$ and $L = \lceil2 \log \log P \rceil$.  Then
\[
\textsf{r}_A\left( P^{1 - \epsilon^6/(10L\,t_n^2 \log t_n)}\right) \leq  P^{5 \epsilon} \,.
\]
\end{lemma}
\begin{proof}
First, we consider the case when there exists an integer $B$ such that  $B \leq t_1^{a_1}, \dots , t_n^{a_n} \leq B^2$.  Note that %
\[(\underbrace{\DFT_{t_i} \otimes \cdots \otimes \DFT_{t_i}}_\text{$a_i$}) = H_{t_i,a_i}\,.\]  %
By {Theorem} \ref{Hadamard}, for each $i$ there exists a matrix $E_i$ such that $E_i$ has at most $t_i^{\epsilon a_i}$ nonzero entries in each row and column and 
\[
\rank(H_{t_i,a_i} - E_i) \leq t_i^{a_i\left(1 - \epsilon^4/(t_i^2 \log t_i) \right)} \,.
\]  
Let $A_i  = H_{t_i,a_i} - E_i$.  Then
\begin{align*}
&(\underbrace{\DFT_{t_1} \otimes \cdots \otimes \DFT_{t_1}}_\text{$a_1$}) \otimes \cdots \otimes  (\underbrace{\DFT_{t_n} \otimes \cdots \otimes \DFT_{t_n}}_\text{$a_n$}) = (E_1 + A_1) \otimes \cdots \otimes (E_n + A_n) \\ &= \sum_{S \subset [n]}\left(\bigotimes_{i \in S} A_i \right) \otimes \left(\bigotimes_{i' \notin S} E_{i'} \right)  
= \sum_{S \subset [n], |S| \geq \epsilon n}\left(\bigotimes_{i \in S} A_i \right) \otimes \left(\bigotimes_{i' \notin S} E_{i'} \right) + \sum_{S \subset [n], |S| < \epsilon n}\left(\bigotimes_{i \in S} A_i \right) \otimes \left(\bigotimes_{i' \notin S} E_{i'} \right) \,.
\end{align*}
Let the first term above be $N_1$ and the second term be $N_2$.  We bound the rank of $N_1$ and the number of nonzero entries in each row and column of $N_2$.  Note that by grouping terms in the sum for $N_1$, we can find matrices $E_S$ for all $S \subset [n]$ with $|S| = \epsilon n$ and write
\[
N_1 = \sum_{S \subset [n], |S| = \epsilon n}\left(\bigotimes_{i \in S} A_i \right) \otimes E_S \,.
\]
Now we have
\begin{eqnarray*}
\rank(N_1) \leq \sum_{S \subset [n], |S| = \epsilon n} P\prod_{i \in S}\frac{1}{t_i^{a_i \epsilon^4/(t_i^2 \log t_i)}} \leq \binom{n}{\epsilon n}\frac{P}{\left(B^{\epsilon^4/(t_n^2 \log t_n)}\right)^{\epsilon n}} \\  \leq \frac{(n)^{\epsilon n}}{\left(\frac{\epsilon n}{3}\right)^{\epsilon n}} \frac{P}{\left(B^{\epsilon^4/(t_n^2 \log t_n)}\right)^{\epsilon n}}\leq \left(\frac{3}{\epsilon}\right)^{\epsilon n} \frac{P}{B^{\epsilon^5 n/(t_n^2 \log t_n)}} \,.
\end{eqnarray*}
Since we assumed $a_i \geq \max\left(\frac{t_i^2(\log t_i)^2}{\epsilon^{10}}, D \right)$, we get
\[
B^{\epsilon^4/(t_n^2 \log t_n)} \geq t_n^{a_n \epsilon^4/(2 t_n^2 \log t_n)} \geq \max \left(t_n^{0.5 \log t_n},t_n^{D \epsilon^4/(2 t_n^2 \log t_n)} \right) \,.
\]
Either the first term is larger than $(3/\epsilon)^2$ or $t_n$ is bounded above by some function of $\epsilon$ in which case if we choose $D$ sufficiently large, the second term will be larger than $(3/\epsilon)^2$.  In any case we get
\[
\rank(N_1) \leq \frac{P}{\left(\frac{\epsilon}{3} \cdot  B^{\epsilon^4/(t_n^2 \log t_n)}\right)^{\epsilon n}} \leq \frac{P}{B^{\epsilon^5 n/(2t_n^2 \log t_n)}} \leq P^{1 - \epsilon^5/(4t_n^2 \log t_n)} \,.
\]
Now we bound the number of nonzero entries in each row and column of $N_2$.  This number is at most
\[
2^nB^{2 \epsilon n}P^{\epsilon} \leq 2^n P^{3\epsilon} \leq P^{4\epsilon} \,.
\]
Thus, when we have $B \leq t_1^{a_1}, \dots , t_n^{a_n} \leq B^2$, 
\[
\textsf{r}_A\left( P^{1 - \epsilon^5/(4t_n^2 \log t_n)} \right) \leq  P^{4\epsilon} \,.
\]

Now we move on to the case where we no longer have control over the range of values $t_1^{a_1}, \dots , t_n^{a_n}$.  Fix $k = 2^D$ and consider the intervals $I_1 = [k,k^2), I_2 = [k^2, k^4), \dots , I_j = [k^{2^{j-1}}, k^{2^j}), \dots $ and so on.  Note 
\[
A = \bigotimes_{i \in [L]}\left(\bigotimes_{t_j^{a_j} \in I_i} (\underbrace{\DFT_{t_j} \otimes \cdots \otimes \DFT_{t_j}}_\text{$a_j$})\right) \,.
\]
For an integer $i$, let $P_i = \prod_{t_j^{a_j} \in I_i}t_j^{a_j}$.  Let $T$ be the set of indices $i \in [L]$ such that $P_i \geq P^{\epsilon/(2L)}$.  Then 
\[
A = \left( \bigotimes_{i \in T}\left(\bigotimes_{t_j^{a_j} \in I_i} (\underbrace{\DFT_{t_j} \otimes \cdots \otimes \DFT_{t_j}}_\text{$a_j$})\right) \right) \otimes \left( \bigotimes_{i \notin T}\left(\bigotimes_{t_j^{a_j} \in I_i} (\underbrace{\DFT_{t_j} \otimes \cdots \otimes \DFT_{t_j}}_\text{$a_j$})\right) \right) = B \otimes C
\]

\noindent where naturally $B$ denotes the first term and $C$ denotes the second.

Note that the dimension of the matrix $C$, which we denote by $|C|$, is at most $\left( P^{\epsilon/(2L)} \right)^L = P^{\epsilon/2}$.  We now apply {Lemma} \ref{products} repeatedly to bound the rigidity of $B$.  Let %
\[B_i = \left(\bigotimes_{t_j^{a_j} \in I_i} (\underbrace{\DFT_{t_j} \otimes \cdots \otimes \DFT_{t_j}}_\text{$a_j$})\right) \,.\]
Then we have,
\[
\textsf{r}_B\left( \left(\prod_{i \in T}P_i \right) \left(\sum_{i \in T} \frac{1}{P_i^{\epsilon^5/(4t_n^2 \log t_n)}}\right)\right) \leq \left(\prod_{i \in T}P_i\right)^{4 \epsilon} \,.
\]
{}From the above inequality, the fact that $P_i \geq P^{\epsilon/(2L)}$ for all $i \in T$, and $|C| \leq P^{\epsilon/2}$ we deduce
\[
\textsf{r}_A\left( P^{1 - \epsilon^6/(10L\,t_n^2 \log t_n)}\right) \leq \textsf{r}_A\left( LP^{1 - \epsilon^6/(8L\,t_n^2 \log t_n)}\right)  \leq |C|\textsf{r}_B\left( \frac{ LP^{1 - \epsilon^6/(8L\,t_n^2 \log t_n)}}{|C|}\right)   \leq P^{5 \epsilon} \,. \qedhere
\]
\end{proof}

\subsection{Proof of {Theorem} \ref{main}}
To complete the proof of {Theorem} \ref{main}, we will break %
$\DFT_N$ %
into submatrices, show that each submatrix is non-rigid using techniques from the previous section, and then combine our estimates to conclude that %
$\DFT_N$ %
 is non-rigid.  Recall that $N$ is $(l,x)$-factorable with 
 \[\frac{x}{(\log x)^{C_0 + 100}} \leq l \leq\frac{x}{(\log x)^{C_0+10}} \,,\] meaning $N = q_1q_2 \cdots q_l$ for some distinct primes $q_1, \dots , q_l$ where $q_i - 1$ has no large prime power divisors for all $i$.  Let $\gamma$ be a primitive $N$\ts{th} root of unity.
\begin{define}\label{def_multandfact}
For a subset $S \subset [l]$ define $\mult_N(S) =  \prod_{s \in S} q_s$ and $\fact_N(S) = \prod_{s \in S} (q_s-1)$.
\end{define} 
\begin{define} For all $S \subset [l]$ we will define $T_S$ as the subset of $[N] \times [N]$ indexed by $(i,j)$ such that 
\begin{align*}
\forall s \in S \qquad ij &\not\equiv 0 \mod q_s \,, \\
\forall s \notin S \qquad ij &\equiv 0 \mod q_s  \,.
\end{align*}
Note that as $S$ ranges over all subsets of $[l]$, the sets $T_S$ form a partition of $[N] \times [N]$.
\end{define}
For each $S$, we will divide the set $T_S$ into submatrices such that when filled with the corresponding entries of %
$\DFT_N$ %
, we can apply {Lemma} \ref{productbound} to show that each submatrix is nonrigid.  The key intuition is that for a given prime $q_i$, once we restrict to nonzero residues, the multiplicative subgroup actually has the additive structure of $\Z_{q_i - 1}$.  Since $q_i - 1$ is smooth, $\Z_{q_i - 1}$ is a direct sum of cyclic groups of small order.

\begin{define}\label{def_factormatrix}
For all $S \subset [l]$, we define the $\fact_N(S) \times \fact_N(S)$ matrix $M(S)$ as follows.  Let $R_S$ be the set of residues modulo $ \mult_N(S)$ that are relatively prime to $\mult_N(S)$.  Note that $|R_S| = \fact_N(S)$.  Each row and each column of $M(S)$ is indexed by an element of $R_S$ and the entry in row $i$ and column $j$ is $\theta^{i \cdot j}$ where $\theta$ is a primitive $\mult_N(S)$ root of unity.  The exact order of the rows and columns will not matter for our uses.  Note that replacing $\theta$ with $\theta^k$ for $k$ relatively prime to $\mult_N(S)$ simply permutes the rows so it does not matter which root of unity we choose.
\end{define}

\begin{lemma}\label{division}
Consider the set of entries in %
$\DFT_N$ %
 indexed by elements of $T_S$.  We can partition this set into $\prod_{s \notin S}(2q_s-1)$ submatrices each of size $\fact_N(S) \times \fact_N(S)$ that are equivalent to $M(S)$ up to some permutation of rows and columns.
\end{lemma}
\begin{proof}
In $T_S$, for each prime $q_s$ with $s \notin S$, there are $2q_s - 1$ choices for what $i$ and $j$ are $\mod q_s$.  Now fix the choice of $i,j \mod q_s$ for all $s \notin S$.  
We %
restrict to indices with $i \equiv c_1 \mod \prod_{s \notin S}q_s$ and $j \equiv c_2 \mod \prod_{s \notin S}q_s$
for some $c_1,c_2$. %

We are left with a $\fact_N(S) \times \fact_N(S)$ matrix, call it $A$, where $i$ and $j$ run over all residues modulo $\mult_N(S)$ that are relatively prime to $\mult_N(S)$.  Naturally, label all rows and columns of this matrix by what the corresponding indices $i$ and $j$ are modulo $ \mult_N(S)$.  For a row labeled $a$ and a column labeled $b$, we compute the entry $A_{ab}$.  The value is $\gamma^{a' \cdot b'}$ where $a'$ is the unique element of $\Z_N$ such that $a' \equiv a \mod \mult_N(S)$ and $a' \equiv c_1 \mod  \prod_{s \notin S}q_s$ and $b'$ is defined similarly.  We have
\begin{align*}
  a' \cdot b' &\equiv ab \mod \mult_N(S)\,,\\[1ex]
  a' \cdot b' &\equiv c_1c_2 \equiv 0 \mod \prod_{s \notin S}q_s \,.
\end{align*}
Therefore
\[
a'b' \equiv k\prod_{s \notin S}q_s ab \mod \mult_N(S)
\]
where $k$ is defined as an integer such that $k\prod_{s \notin S}q_s \equiv 1 \mod \mult_N(S)$.  Note that $k$ clearly exists since $\prod_{s \notin S}q_s$ and $\mult_N(S)$ are relatively prime.  Since $\gamma^{k\prod_{s \notin S}q_s}$ is a primitive $\mult_N(S)$ root of unity, the matrix $A$ is equivalent to $M(S)$ up to some permutation, as desired.
\end{proof} 

\begin{lemma}\label{blockrigidity}
For a subset $S \subset [l]$ with $|S| = k$ and $M(S)$ (as defined in {Definition} \ref{def_factormatrix})  a  $\fact_N(S) \times \fact_N(S)$ matrix as described above.  we have 
\[
\textsf{r}_{M(S)}\left( \frac{\fact_N(S)}{\exp\left(\epsilon^6 x^{0.37}\right)} \right) \leq \left(\fact_N(S)\right)^{6 \epsilon}
\]
as long as $k \geq \frac{x}{(\log x)^{C_0 + 200}}$.
\end{lemma}
\begin{proof}
Without loss of generality $S = \{1,2, \dots , k \}$.  Consider the factorizations of $q_1-1, \dots , q_k-1$ into prime powers.  For each prime power $p_i^{e_i} \leq x^{0.3}$, let $c(p_i^{e_i})$ be the number of indices $j$ for which $p_i^{e_i}$ appears (exactly) in the factorization of $q_j-1$.  Consider all prime powers $p_i^{e_i}$ for which $c(p_i^{e_i}) < x^{0.62}$.
\[
 \prod_{t, c(t) \leq x^{0.62}} t^{c(t)} \leq \left((x^{0.3})^{x^{0.62}} \right)^{x^{0.3}} \leq x^{x^{0.92}} \,.
\]
Now consider all prime powers say $t_1, \dots , t_n$ for which $c(t_i) \geq x^{0.62}$.  Let $P = t_1^{c(t_1)} \cdots t_n^{c(t_n)}$.  {}From the above we know that as long as $x$ is sufficiently large
 \begin{equation}\label{separateprimepowers}
P \geq \frac{\fact_N(S)}{x^{x^{0.92}} } \geq \left(\fact_N(S)\right)^{(1 - \epsilon)} \frac{\left(\frac{x}{(\log x)^{C_0 + 1}} \right)^{\epsilon k}}{x^{x^{0.92}} } \geq \left(\fact_N(S)\right)^{(1 - \epsilon)} \,.
 \end{equation}

We will use the prime powers $t_i$ and \expref{Theorem}{Hadamard} to show that $M(S)$ is not rigid.  Note that we can associate each row and column of $M(S)$ %
with  %
an ordered $k$-tuple $(a_1, \dots , a_k)$ where $a_i \in \Z_{q_i-1}$ as follows. First, it is clear that each row and column of $M(S)$ can be associated
with  %
an ordered $k$-tuple 
$(z_1, \dots , z_k) \in \F_{q_1}^{\times} \times \cdots \times \F_{q_k}^{\times}$.
Now $\Z_{q_i}^{\times}$ can be viewed as a cyclic group on $q_i-1$ elements.  This allows us to create a bijection between the rows and columns of $M(S)$ and elements of $\Z_{q_1-1} \times \cdots \times \Z_{q_k-1}$.

Also note that for a row indexed by $A = (a_1, \dots , a_k)$ and a column indexed by $B = (b_1, \dots , b_k)$, the entry $M(S)_{AB}$ is dependent only on $A+B$.  We will now decompose $M(S)$ into several $P \times P$ submatrices.  In particular, we can write $q_i-1 = d_i T_i$ where $T_i$ is a product of some subset of $\{t_1, \dots , t_n \}$ and $d_i$ is relatively prime to $T_i$.  We have $T_1T_2\cdots T_k = P$.  For each $A',B' \in \Z_{d_1} \times \cdots \times \Z_{d_k}$, we can construct a $P \times P$ submatrix $M(S,A',B')$ consisting of all entries $M(S)_{AB}$ of $M(S)$ such that $A \equiv A', B \equiv B'$ (where the equivalence is over $\Z_{d_1} \times \cdots \times \Z_{d_k}$).  This gives us $d^2$ different submatrices where $d =  d_1\cdots d_k$.  Naturally, we can associate each row and column of a submatrix $M(S,A',B')$ with an element of $\Z_{T_1} \times \cdots \times \Z_{T_k}$ such that for a row labeled $I$ and a column labeled $J$, the entry $M(S,A',B')_{IJ}$ only depends on $I+J$.  In particular, this means that $X\left(M(S,A',B')\right)X$ is diagonal where %
$X = \DFT_{T_1} \otimes \cdots \otimes \DFT_{T_k}$ %
.  Now, using \expref{Lemma}{decomp}, we can rewrite 
\[
X = (\underbrace{\DFT_{t_1} \otimes \cdots \otimes \DFT_{t_1}}_\text{$c(t_1)$}) \otimes \cdots \otimes  (\underbrace{\DFT_{t_n} \otimes \cdots \otimes \DFT_{t_n}}_\text{$c(t_n)$}) \,.
\] %

Since for $x$ sufficiently large, $c(t_i) \geq x^{0.62} \geq t_i^2(\log t_i)^2/\epsilon^{10}$, we can use \expref{Lemma}{productbound} and get that
\[
\textsf{r}_X\left( P^{1 - \epsilon^6/(20(\log \log P) x^{0.62})}\right) \leq  P^{5 \epsilon} \,.
\]

Let $E$ be the matrix of changes to reduce the rank of $X$ according to the above.  We have that $E$ has at most $P^{\epsilon}$ nonzero entries in each row and column, and  
\[
\rank(X - E) \leq P^{1 - \epsilon^6/(20(\log \log P) x^{0.62})} \,.
\]
We can write $M(S)$ in block form as 
\[
\begin{bmatrix}
M(S,A_1,B_1) &  M(S,A_1, B_2) &  \dots & M(S, A_1, B_d)\\
M(S, A_2, B_1) & M(S, A_2, B_2) & \dots &  M(S, A_1, B_d)\\
\vdots & \vdots & \ddots & \vdots \\
M(S, A_d, B_1) & M(S, A_d, B_2) & \dots & M(S, A_d, B_d)
\end{bmatrix}
\]
where $A_1, \dots , A_d$ and $B_1, \dots , B_d$ range over the elements of $\Z_{d_1} \times \cdots \times \Z_{d_k}$.  We can rearrange the above as 
\[
\begin{bmatrix}
M(S,A_1,B_1) &  \dots & M(S, A_1, B_d)\\
\vdots  & \ddots & \vdots \\
M(S, A_d, B_1)  & \dots & M(S, A_d, B_d)
\end{bmatrix} = 
\begin{bmatrix}
XD_{11}X &  \dots & XD_{1d}X\\
\vdots  & \ddots & \vdots \\
XD_{d1}X  & \dots & XD_{dd}X
\end{bmatrix}
\]
where the $D_{ij}$ are diagonal matrices.  Now consider the matrix
\[
E(S) = 
\begin{bmatrix}
 ED_{11}E &  \dots & ED_{1d}E\\
\vdots  & \ddots & \vdots \\
ED_{d1}E  & \dots & ED_{dd}E
\end{bmatrix} \,.
\]
We have
\begin{align*}
&M(S) - E(S) = 
\begin{bmatrix}
XD_{11}X - ED_{11}E &  \dots & XD_{1d}X - ED_{1d}E\\
\vdots  & \ddots & \vdots \\
XD_{d1}X - ED_{d1}E  & \dots & XD_{dd}X - ED_{dd}E
\end{bmatrix}
= \\
&\begin{bmatrix}
XD_{11}(X - E)  &  \dots & XD_{1d}(X - E) \\
\vdots  & \ddots & \vdots \\
XD_{d1}(X - E)    & \dots & XD_{dd}(X - E)  
\end{bmatrix}  + 
\begin{bmatrix}
 (X-E)D_{11}E & \dots &  (X-E)D_{1d}E \\
 \vdots  & \ddots & \vdots  \\
 (X-E)D_{d1}E & \dots & (X-E)D_{dd}E
\end{bmatrix} \,.
\end{align*}

In the above expression, each of the two terms has rank at most 
\[
dP^{1 - \epsilon^6/(20(\log \log P) x^{0.62})} = \frac{\fact_N(S)}{P^{\epsilon^6/(20(\log \log P) x^{0.62})}} \leq \frac{1}{2}\left( \frac{\fact_N(S)}{\exp\left(\epsilon^6 x^{0.37}\right)} \right) \,.
\]
Note that when computing the rank, we only multiply by $d$ (and not $d^2$) because the small blocks are all multiplied by the same low rank matrix on either the left or right. The number of nonzero entries in each row and column of $E(S)$ is at most $P^{5\epsilon}d = \frac{\fact_N(S)}{P^{1-5\epsilon}}$.  Since $P \geq \left( \fact_N(S) \right)^{1 - \epsilon}$,  we conclude 
\[
\textsf{r}_{M(S)}\left( \frac{\fact_N(S)}{\exp\left(\epsilon^6 x^{0.37}\right)} \right) \leq \left(\fact_N(S)\right)^{6 \epsilon} \,. \qedhere
\]
\end{proof}

We are now ready to complete the analysis of the non-rigidity of the %
DFT %
matrix %
$\DFT_N$ %
.
\begin{proof}[Proof of \expref{Theorem}{main}]
Set the threshold $m = x^{0.365}$ and $k_0 = l -m$.  The sets $T_S$, as $S$ ranges over all subsets of $[l]$, form a partition of $[N] \times [N]$.  For each $S \subset [l]$ with $|S| \geq k_0$, we will divide $T_S$ into $\fact_N(S) \times \fact_N(S)$ submatrices using \expref{Lemma}{division} and change entries to reduce the rank of every submatrix according to \expref{Lemma}{blockrigidity}.  We will not touch the entries in sets $T_S$ for $|S| < k_0$.  Call the resulting matrix $M'$.  We now estimate the rank of $M'$ and then the maximum number of entries changed in any row or column.

We remove all rows and columns corresponding to integers divisible by at least $\frac{m}{2}$ of the primes $q_1, \dots , q_l$.  The number of rows and columns removed is at most
\[
N\left( \sum_{S \subset [l] , |S| = \frac{m}{2}} \prod_{i \in S} \frac{1}{q_i} \right) \leq \frac{N}{\left(\frac{x}{(\log x)^{C_0 }}\right)^{m/2}} \binom{l}{m/2} < N \left(\frac{l}{\frac{x}{(\log x)^{C_0 }}}\right)^{m/2}\leq \frac{N}{(\log x)^{x^{0.365}}} \,.
\]

The remaining entries must be subdivided into matrices of the form $M(S)$ for various subsets $S \subset [l]$ %
with %
 $|S| \geq k_0$.  
Let   %
 $q_1 < q_2 < \dots < q_l$.  The number of such submatrices is at most 
\[
\begin{multlined}
\frac{N^2}{\left( (q_1-1) \cdots (q_{k_0}-1) \right)^2} \leq (q_{k_0+1} \cdots q_l)^2 \left( \frac{q_1 \cdots q_{k_0}}{(q_1-1) \cdots (q_{k_0}-1)} \right)^2 \leq 3(q_{k_0+1} \cdots q_l)^2 \leq 3x^{2m} \,.
\end{multlined}
\]
Each one of the submatrices has rank at most 
\[
\frac{N}{\exp\left(\epsilon^6 x^{0.37}\right)} \,,
\]
so in total the rank is at most
\[
N\frac{3x^{2m}}{\exp\left(\epsilon^6 x^{0.37}\right)} \leq \frac{N}{\exp\left(\epsilon^6 x^{0.369}\right)} \,.
\]
Combining the two parts we easily get
\[
\rank(M') \leq \frac{N}{\exp\left(\epsilon^6 x^{0.365}\right)} \,.
\]

Now we bound the number of entries changed.  The number of entries changed in each row or column is at most
\[
\begin{multlined}
\frac{N}{\left( (q_1-1) \cdots (q_{k_0}-1) \right)}N^{ 6\epsilon} \leq  (q_{k_0+1} \cdots q_l) \left(\frac{q_1 \cdots q_{k_0}}{(q_1-1) \cdots (q_{k_0}-1)} \right) N^{ 6\epsilon} 
\leq 3N^{6\epsilon + 1.1 m/l} \leq N^{ 7 \epsilon} \,.
\end{multlined}
\]
As $\exp\left(\epsilon^6x^{0.365}\right) \geq \exp\left(\epsilon^6 (\log N)^{0.36}\right)$ for sufficiently large $x$, we conclude
\[
\textsf{r}_{\DFT_N}\left( \frac{N}{\exp\left(\epsilon^6 (\log N)^{0.36}\right)} \right) \leq N^{ 7 \epsilon} \,. \qedhere
\] %
\end{proof}

\section{Non-rigidity of all circulant matrices}\label{sec_Fourier2}
In the previous section, we showed that there exists an infinite set of %
DFT %
 matrices that are not Valiant-rigid.  In this section, we will bootstrap the results from \expref{Section}{sec_Fourier1} to show that in fact, %
no (sufficiently large) DFT matrix is rigid.

The first ingredient will be a stronger form of \expref{Lemma}{construction}.  Recall that a prime $q$ is defined to be $x$-good if $x/(\log x)^{C_0} \leq q \leq x$ and all prime powers dividing $q-1$ are at most $x^{0.3}$ and that an integer $N$ is defined to be $(l,x)$-factorable if it can be written as the product of $l$ distinct $x$-good primes.

To simplify our formulas we use the following notation.   %
\begin{notation}    \label{not:bypolylog}
Define 
\[
g_k(x) = \frac{x}{(\log x)^{C_0 + k}} \,.
\]
\end{notation}

\begin{lemma}\label{scales}
For all sufficiently large integers $K$, there exist $l,x,N$ such that %
the following conditions hold:
\begin{itemize}
\item $g_{100}(x) \leq l \leq g_{10}(x)$,
\item $N$ is $(l,x)$-factorable,
\item$K < N < K (\log K)^2$.
\end{itemize}
\end{lemma}
\begin{proof}

Call an $N$ well-factorable if it is $(l,x)$-factorable for some $x$ and $g_{100}(x) \leq l \leq g_{10}(x)$.  Let $N_0$ be the largest integer that is well-factorable with $N_0 \leq K$.  %
Assume %
$N_0$ is $(l ,x)$-factorable.

We have $N_0 = q_1 \cdots q_l$ where $q_1, \dots , q_l$ are distinct, $x$-good primes.  If $l < \lfloor g_{10}(x) \rfloor$ then by \expref{Lemma}{construction}, we can find another $x$-good prime $q_{l+1}$.  We can then replace $N_0$ with $q_{l+1}N_0$.  $q_{l+1}N_0 > K$ by the maximality of $N_0$ and also $q_{l+1}N_0 \leq N_0x \leq N_0 (\log N_0)^2$ so $q_{l+1}N_0$ satisfies the desired 
conditions. %

We now consider the case where $l = \lfloor g_{10}(x) \rfloor$.  First, if $q_1, \dots , q_l$ are not the $l$ largest $x$-good primes then we can replace one of them say $q_1$ with $q_1'>q_1$.  The number $N' = q_1'q_2 \cdots q_l$ is well-factorable and between $N_0$ and $N_0(\log x)^{C_0}$.  Using the maximality of $N_0$, we deduce that $N'$ must be in the desired range.  

On the other hand if $q_1, \dots , q_l$ are the $l$ largest $x$-good primes, we know they are actually all between $3x/(\log x)^{C_0}$ and $x$.  This is because by \expref{Lemma}{construction}, there are at least $10x/(\log x)^{C_0}$ distinct $x$-good primes.  Let $x' = 2x$.  The above implies that $q_1, \dots , q_l$ are $x'$-good and clearly $g_{100}(x') \leq l \leq g_{10}(x')$.  Furthermore, $g_{10}(x') >g_{10}(x) + 1$ so $l = \lfloor g_{10}(x) \rfloor <  \lfloor g_{10}(x') \rfloor$ and we can now repeat the argument from the first case.
\end{proof}

We can now complete the proof that %
 circulant matrices are not rigid. 
\begin{theorem}\label{main_circulant}
Let $0 < \epsilon < 0.01$ be a given parameter.  For all sufficiently 
large $N$, if $M$ is an $N \times N$ %
circulant (or Toeplitz) %
matrix, then %
\[
\textsf{r}_{M}\left(\frac{N}{\exp\left(\epsilon^6(\log N)^{0.35}\right)}\right) \leq N^{15 \epsilon} \,.
\]
\end{theorem}
\begin{proof}
First we analyze circulant matrices of size $N_0$ where $N_0$ is $(l,x)$-factorable for some $g_{100}(x) \leq l \leq g_{10}(x)$.  \expref{Theorem}{main} and \expref{Lemma}{diagonalization} imply that for $M_0$ an $N_0 \times N_0$ circulant matrix where $N_0$ satisfies the previously mentioned 
conditions, %
\[
\textsf{r}_{M_0}\left(\frac{2N_0}{\exp\left(\epsilon^6(\log N_0)^{0.36}\right)}\right) \leq N_0^{14 \epsilon} \,.
\]
Now for a circulant matrix $M$ of arbitrary size $N \times N$, note that it is possible to embed an $M$ in the upper left corner of a circulant matrix of any size at least $2N$.  By \expref{Lemma}{scales}, there exists an $N_0$ that is $(l,x)$-factorable for some $g_{100}(x) \leq l \leq g_{10}(x)$ such that
\[
\frac{N_0}{(\log N_0)^2} \leq N \leq \frac{N_0}{2} \,.
\]
We deduce
\[
\textsf{r}_{M}\left(\frac{2N_0}{\exp\left(\epsilon^6(\log N_0)^{0.36}\right)}\right) \leq N_0^{14 \epsilon} \,.
\]
Rewriting the bounds in terms of $N$ we get
\[
\textsf{r}_{M}\left(\frac{N}{\exp\left(\epsilon^6(\log N)^{0.35}\right)}\right) \leq N^{15 \epsilon} \,. \qedhere
\]
\end{proof}
\begin{remark}
Note that our proof actually shows something slightly stronger, namely that the changes to reduce the rank of a circulant matrix are actually fixed linear combinations of the entries.  See \expref{Definition}{group_reducibility} and \expref{Claim}{main_stronger} for a more precise statement.
\end{remark}
{}From the above and \expref{Claim}{rescaling}, we immediately deduce that %
DFT %
 matrices are not rigid.
\begin{theorem}\label{full_fourier}
Let $0 < \epsilon < 0.01$ be a given parameter.  For all sufficiently large $N$,
 \[
\textsf{r}_{\DFT_N}\left(\frac{N}{\exp\left(\epsilon^6(\log N)^{0.35}\right)}\right) \leq N^{15 \epsilon} \,.
\]
\end{theorem}

\section{Non-rigidity of
 $G$-circulant matrices for abelian groups}\label{sec_groupalgebra}
Using the results from the previous section, we can show that %
$\DFT_G$ and %
$G$-circulant matrices are not Valiant-rigid for any
infinite class of finite abelian groups $G$.  %
Our proof follows the same strategy as
the proof of \expref{Lemma}{productbound}.

\begin{theorem}\label{DFTabelian}
Let $0 < \epsilon < 0.01$ be fixed.  Let $G$ be an abelian group and $f: G \rightarrow \C$ be a function.  If $|G|$ is sufficiently large then
\[
\textsf{r}_{\DFT_G}\left(\frac{|G|}{\exp\left(\epsilon^8 (\log |G|)^{0.32}\right)}\right) \leq  |G|^{19 \epsilon} \,.
\]
\end{theorem}
\begin{proof}
By the Fundamental Theorem of Finite Abelian Groups we can write 
$G = \Z_{n_1} \times \cdots \times \Z_{n_a}$\,.
By \expref{Fact}{fact:dftproduct},
we have $\DFT_G = \DFT_{n_1} \otimes \cdots \otimes \DFT_{n_a}$.  
Let us write $F = \DFT_G$. 
 
Without loss of generality, $n_1 \leq n_2 \leq \dots \leq n_a$.  We will choose $k$ to be a fixed, sufficiently large positive integer.  By \expref{Theorem}{full_fourier}, we can ensure that for $N > k$
\[
\textsf{r}_{\DFT_N}\left(\frac{N}{\exp\left(\epsilon^6 (\log N)^{0.35}\right)} \right) \leq N^{15 \epsilon} \,.
\]
Consider the ranges $I_1 = [k,k^2), I_2 = [k^2, k^4), \dots I_j = [k^{2^{j-1}},k^{2^j}) \dots $ and so on.  Let $S_j$ be a multiset defined by $S_j = I_j \cap \{n_1, \dots, n_a \}$.  Fix a $j$ and let the elements of $S_j$ be $x_1 \leq \dots \leq x_b$.  By \expref{Theorem}{full_fourier}, for each $x_i$, there are matrices $E_{x_i}$ and $A_{x_i}$ such that %
$\DFT_{x_i} = A_{x_i} + E_{x_i}$ %
, $E_{x_i}$ has at most $x_i^{15\epsilon}$ nonzero entries in each row and column, and 
\[
\rank(A_{x_i}) \leq \frac{x_i}{\exp\left(\epsilon^6 (\log x_i)^{0.35}\right)} \,.
\]
Now we can write

\begin{align*}
M_j = \DFT_{x_1} \otimes \cdots \otimes \DFT_{x_b}   = (A_{x_1} + E_{x_1}) \otimes \cdots \otimes (A_{x_b} + E_{x_b}) = \sum_{S \subset [b]}\left(\bigotimes_{i \in S}A_{x_i}\right) \otimes \left(\bigotimes_{i' \notin S}E_{x_i'}\right) \\ = \sum_{S \subset [b], |S| \geq \epsilon b}\left(\bigotimes_{i \in S}A_{x_i}\right)\otimes \left(\bigotimes_{i' \notin S}E_{x_i'} \right)+ \sum_{S \subset [b], |S| < \epsilon b}\left(\bigotimes_{i \in S}A_{x_i}\right)\otimes \left(\bigotimes_{i' \notin S}E_{x_i'} \right) \,.
\end{align*}
Let the first term above be $N_1$ and the second term be $N_2$.  We will bound the rank of $N_1$ and the number of nonzero entries in each row and column of $N_2$. Note that by grouping the terms in the sum for $N_1$ we can write it in the form
\[
\sum_{S \subset [b], |S| = \lceil \epsilon b \rceil}\bigotimes_{i \in S}A_{x_i} \otimes E_S
\]
where %
$E_S$ is some matrix for each $S$. %
 This implies that 
\begin{align*}
\rank(N_1) \leq \binom{b}{\lceil \epsilon b \rceil} \frac{x_1 \cdots x_b}{\left(\exp\left(\epsilon^6 (\log x_1)^{0.35}\right)\right)^{\lceil \epsilon b \rceil}} \leq \frac{b^{\lceil \epsilon b \rceil}}{( \frac{\epsilon b}{3})^{\lceil \epsilon b \rceil}}\frac{x_1 \cdots x_b}{\left(\exp\left(\epsilon^6 (\log x_1)^{0.35}\right)\right)^{\lceil \epsilon b \rceil}} \\ = x_1 \cdots x_b\left( \frac{3}{\epsilon \exp\left(\epsilon^6 (\log x_1)^{0.35}\right)} \right)^{\lceil \epsilon b \rceil} \,.
\end{align*}
As long as $k$ is sufficiently large, we have
\begin{eqnarray*}
\rank(N_1) \leq x_1 \cdots x_b\left( \frac{3}{\epsilon \exp\left(\epsilon^6 (\log x_1)^{0.35}\right)} \right)^{\lceil \epsilon b \rceil} \leq x_1 \cdots x_b\left( \frac{1}{\exp\left(\epsilon^6 (\log x_1)^{0.34}\right)} \right)^{\lceil \epsilon b \rceil}  \\ \leq \frac{ x_1 \cdots x_b}{\exp\left(\epsilon^7 (\log x_1 \cdots x_b )^{0.33}\right)}
\end{eqnarray*}
where in the last step we used the fact that $x_i \leq x_1^2$ for all $i$.  The number of nonzero entries in each row or column of $N_2$ is at most 
\[
2^b x_b \cdots x_{b - \lfloor \epsilon b\rfloor + 1}(x_{b - \lfloor \epsilon b\rfloor} \cdots x_1)^{15\epsilon} = 2^b (x_1 \cdots x_b)^{15\epsilon}(x_b \cdots x_{b - \lfloor \epsilon b\rfloor + 1})^{1 - 15\epsilon} \leq (x_1 \cdots x_b)^{18\epsilon} \,.
\]
Note in the last step above, we used the fact that $x_i \leq x_1^2$.

For each integer $c$ between $2$ and $k$, let $m_c$ be the number of copies of $c$ in the set $\{n_1, \dots , n_a \}$.  If $m_c \geq k^2 (\log k)^2/\epsilon^4 $ then by \expref{Theorem}{Hadamard}, if we define %
$A_c = \underbrace{\DFT_c \otimes \cdots \otimes \DFT_c}_{\text{$m_c$}}$ %
 then 
\[
\textsf{r}_{A_c}\left(c^{m_c(1 - \epsilon^4/(k^2 \log k))}\right) \leq c^{m_c \epsilon} \,.
\]
Let $L = \lceil 2\log \log |G| \rceil$ and ensure that $|G|$ is sufficiently large so that $L > k$.  Let $T$ be the set of integers $c$ between $2$ and $k$ such that $c^{m_c} \geq |G|^{\epsilon/(2L)}$ (note that as long as $|G|$ is sufficiently large, all elements of $T$ must satisfy $m_c \geq k^2 (\log k)^2/\epsilon^4 $).  Let $R$ be the set of indices $j$ for which $\prod_{x \in S_j}x \geq |G|^{\epsilon/(2L)}$.  Since $S_j$ is clearly empty for $j \geq L$, the matrix $F$ can be written as 
\[
F = \left( \bigotimes_{2 \leq c < k} \left( \underbrace{\DFT_c \otimes \cdots \otimes \DFT_c}_{\text{$m_c$}}\right) \right) \otimes \left(\bigotimes_{1 \leq j \leq L} M_j \right) \,.
\]
Define 
\[B = \left( \bigotimes_{c \notin T} \left( \underbrace{\DFT_c \otimes \cdots \otimes \DFT_c}_{\text{$m_c$}}\right) \right) \otimes \left(\bigotimes_{j \notin R} M_j \right) \,. \]
Note that the size of $B$ is at most 
\[
\left( |G|^{\epsilon/(2L)}\right)^{k + L} \leq |G|^{\epsilon} \,.
\]
Also $F = B \otimes D$ where
\[
D =  \left( \bigotimes_{c \in T} \left( \underbrace{\DFT_c \otimes \cdots \otimes \DFT_c}_{\text{$m_c$}}\right) \right) \otimes \left(\bigotimes_{j \in R} M_j \right) \,.
\]
For any rank $r$, %
we have %
$\textsf{r}_M(|B|r) \leq |B|\textsf{r}_D(r)$.  Applying \expref{Lemma}{products} iteratively, we get 
\begin{align*}
\textsf{r}_D\left(\frac{|G|}{|B|}\left( \sum_{c \in T}\frac{1}{c^{m_c\epsilon^4/(k^2 \log k)}} + \sum_{j \in R}\frac{1}{\exp\left(\epsilon^7 (\log \prod_{x \in S_j}x )^{0.33}\right)} \right)\right) \leq \left(\frac{|G|}{|B|}\right)^{18 \epsilon} \,.
\end{align*}
Note that 
\begin{align*}
\left( \sum_{c \in T}\frac{1}{c^{m_c \epsilon^4/(k^2 \log k)}} + \sum_{j \in R}\frac{1}{\exp\left(\epsilon^7 (\log \prod_{x \in S_j}x )^{0.33}\right)} \right) \leq \frac{k}{|G|^{\epsilon^5/(2Lk^2 \log k)}} + \frac{L}{\exp\left(\epsilon^8 (\log |G|/2L)^{0.33}\right)} \\ \leq \frac{1}{\exp\left(\epsilon^8 (\log |G|)^{0.32}\right)} \,.
\end{align*}
Overall, we conclude
\[
\textsf{r}_F\left(\frac{|G|}{\exp\left(\epsilon^8 (\log |G|)^{0.32}\right)}\right) \leq |B|\left(\frac{|G|}{|B|}\right)^{18 \epsilon} \leq |G|^{19 \epsilon} \,.\qedhere
\]
\end{proof}

\begin{theorem}\label{allabelian}
  Let $0 < \epsilon < 0.01$ be fixed.  Let $G$ be an abelian group and $f: G \rightarrow \C$ be a function. Let $M = M_G(f)$ be a %
$G$-circulant matrix.  If $|G|$ is sufficiently large then
\[
\textsf{r}_M\left(\frac{2|G|}{\exp\left(\epsilon^8 (\log |G|)^{0.32}\right)} \right) \leq |G|^{38 \epsilon} \,.
\]
\end{theorem}
\begin{proof}

Note $\DFT_G$ diagonalizes $M$.  Thus, combining \expref{Theorem}{DFTabelian}
with  \expref{Lemma}{diagonalization} gives the
desired conclusion.
\end{proof}

The rigidity results we proved hold over $\C$.  By examining the proofs more carefully, we can actually show that when $G$ is an abelian group, the same results hold for $G$-circulant matrices over an
abelian   %
extension of $\Q$ of degree $\tilde{O}(N^3)$.
\begin{theorem}\label{thm:extension_degree}
Let $G$ be an abelian group of order $N$.  Then there exists $m = \tilde{O}(N^3)$, depending only on $G$, such that $G$-circulant matrices with entries in $\Q$ satisfy
\[
\textsf{r}_M\left(\frac{2|G|}{\exp\left(\epsilon^8 (\log |G|)^{0.32}\right)} \right) \leq |G|^{38 \epsilon}
\]
over the $m$\ts{th} cyclotomic field.  The same bound holds for the matrix $\DFT_G$.
\end{theorem}
\begin{proof}
First note that it is immediate from our proof that the GWH matrix $H_{d,n}$ is not rigid over $\Q[\omega]$ where $\omega$ is a
$d$\ts{th}  %
primitive root of unity. 
Now for well-factorable integers $N = p_1p_2 \cdots p_k$, the additional roots of unity that we need to adjoin for non-rigidity of $\DFT_N$ are all roots of unity with order dividing $(p_1-1)(p_2-1) \cdots (p_k - 1)$.  Thus, $\DFT_N$ is not rigid over an extension $\Q[\omega][\alpha]$ where $\alpha$
is a root of unity of order at most $N$.  %
Finally for other values of $N$, we find a well-factorable integer $N' = \tilde{O}(N)$ and embed $\DFT_N$ into an $N' \times  N'$ circulant matrix.  It is now immediate from \expref{Lemma}{diagonalization} that $\DFT_N$ is not rigid over a cyclotomic field of order $\tilde{O}(N^3)$.  The proof in \expref{Theorem}{allabelian} for  $\DFT_G$ and $G$-circulant matrices generalizes directly.  
\end{proof}
\section{Finite field case}\label{sec_finitefield}  %
In this section, we sketch how to modify the proofs in the previous sections to deal with matrices over a finite field.  The main difficulty that arises when attempting to extend the above methods to finite fields is that the entries of the corresponding %
DFT %
 matrix might not exist in the field.  Furthermore, for a finite field $\F_q$ and integer $k$ with $\gcd(k,q) > 1$, there are no primitive $k$\ts{th} roots of unity %
in %
any extension of $\F_q$.  Because this section involves a significant amount of abstract algebra, we begin by giving a brief overview of the algebraic tools that we will use.

\subsection{Preliminaries about Galois theory and finite fields}
\label{sec:galois}
Our standard reference for field extensions, Galois theory,
and finite fields is Chapters~V and~VI of
Lang's Algebra~\cite{lang}.  The monograph by
Lidl and Niederreiter~\cite{lidl} is entirely dedicated  %
to finite fields.

\begin{define}
A field extension $\K/\F$ means that $\K$ is a field and $\F$ is
a subfield.  The \emph{degree} of the extension is the dimension
of $\K$ as a vector space over $\F$.  Given a field extension $\K/\F$,
we say that $\alpha \in \K$ is \emph{algebraic} over $\F$ if $\alpha$
is a root of some nonzero polynomial $P$ with coefficients in $\F$.
We say that $\K/\F$ is an \emph{algebraic extension} if every
element of $\K$ is algebraic over $\F$.  A \emph{finite extension}
is an extension of finite degree. 
\end{define}

\begin{fact.}
Every finite extension is algebraic.
\end{fact.}

\begin{define}
Given a field extension $\K/\F$ and $\alpha \in \K$ that is
algebraic over $\F$, we say that the polynomial $P$
over $\F$ is the \emph{minimal polynomial} of $\alpha$
if $P$ is monic and has minimal degree among all nonzero
polynomials over $\F$ that have $\alpha$ as a root.
The \emph{degree} of $\alpha$ over $\F$ is the degree of its
minimal polynomial.
\end{define}
It is not difficult to see that $P$ exists, is
unique, and must be irreducible over $\F$.

\begin{fact.}
Given a field extension $\K/\F$, 
let $\alpha \in \K$ be algebraic over $\F$.
The set $\F[\alpha]$, defined as the set of all polynomials of
$\alpha$ with coefficients in $\F$, is a field.
The degree of the extension $\F[\alpha]/\F$
is the degree of $\alpha$ over $\F$.
\end{fact.}
$\F[\alpha_1,\alpha_2]$ denotes $\F[\alpha_1][\alpha_2]$. 
\begin{define}
Let $\K/\F$ be a field extension 
and let $\alpha \in \K$ be algebraic over $\F$.
Let $P$ be the minimal polynomial of $\alpha$.  
The \emph{conjugates} of $\alpha$ are the roots of $P$
(including $\alpha$ itself)
in an extension of $\K$ over which $P$ decomposes into linear factors. %
\end{define}

\begin{define}[Galois extensions]
  An algebraic field extension $\K/\F$ is \emph{normal} if for
  every irreducible polynomial $P$ over $\F$, if $P$ has a root
  in $\K$ then $P$ splits into linear factors over $\K$.
  The extension $\K/\F$ is \emph{Galois} if it is normal
  and for all $\alpha\in\K$, all roots of the minimal polynomial
  of $\alpha$ over $\F$ are distinct.
\end{define}

\begin{fact.}[{\cite[Ch. V, Thm. 5.5]{lang}}]
If $\K$ is a finite field then every extension $\K/\F$ is Galois.
\end{fact.}

\begin{define}[Galois group]
For a Galois extension $\K/\F$ we write $\Gal(\K/\F)$
to denote the set of those automorphisms of $\K$ that
fix $\F$ elementwise.
\end{define}
  
We begin by stating some basic facts.  

\begin{fact.}   \label{fact:galois-extension-facts}
Let $\K/\F$ be a finite Galois extension of degree $g$
and let $\alpha\in \K$ have degree $m$.  Let $G=\Gal(\K/\F)$.
Then the following hold.
\begin{itemize}
\item[(i)] $|G|=g$.
\item[(ii)] \ $m\mid g$.
\item[(iii)] The conjugates of $\alpha \in \K$ are the elements
  $\pi(\alpha)$ for all $\pi \in \Gal(\K/\F)$.  The
  list $(\pi(\alpha)\mid \pi\in G)$ includes each
  conjugate of $\alpha$ exactly $g/m$ times.
  In particular, if
  $\K = \F[\alpha]$ (i.e., $m=g$) then the degree of
  this extension is the number of conjugates of $\alpha$.
\item[(iv)] $\F$ is precisely the set of common fixed points of
$\Gal(\K/\F)$.
\end{itemize}
\end{fact.}

The following consequence of item (iv) is immediate.
\begin{fact.}\label{symmetric_poly}
Let $\K/\F$ be a Galois extension.
Let $\alpha \in \K$ and let $\alpha_1, \dots , \alpha_m$
be the conjugates of $\alpha$ with $\alpha_1 = \alpha$.  Then
$Q(\alpha_1, \dots , \alpha_m) \in \F$ for any symmetric polynomial $Q$.
\end{fact.}

\begin{fact.}[{\cite[Ch. V, Thm. 5.4]{lang}}]
\label{frobenius}
Let $\K$ be a finite field and $\F$ a subfield.  
If $\F = \F_q$ then $\Gal(\K/\F)$  
is a cyclic group, generated by the Frobenius automorphism
$x \mapsto x^q$. 
\end{fact.}

We have the following consequence.
\begin{fact.} \label{order1}
Let $\K$ be a finite field and $\F = \F_q$ a subfield.  Let $\alpha \in \K$.
Then the conjugates of $\alpha$ over $\F$ are precisely the  
elements of the form $\alpha^{q^j}$ for nonnegative integers $j$.
In particular, if $\K = \F_{q^m} = \F[\alpha]$ then $m$ is the degree
of $\alpha$ over $\F$ and the conjugates of $\alpha$ are exactly
$\{\alpha^{q^0}, \alpha^{q^1},  \dots,  \alpha^{q^{m-1}} \}$.  
Moreover, if $n$ is the order of $\alpha$ in the multiplicative group
of $\K$ then $m = \ord_q(n)$, the order of $q$ modulo $n$.
\end{fact.}

Now we introduce the concept of primitive 
roots of unity over finite fields and prove some
of their basic properties. 
\begin{fact.}[{\cite[Ch. V, Thm. 5.3]{lang}}]
The multiplicative group of a finite field is cyclic.
(In fact, the finite subgroups of the multiplicative group
of any field are cyclic.)
\end{fact.}

\begin{definition}
Let $\F$ be a field.  We say that $\alpha\in\F$ is an
$n$\ts{th} root of unity if $\alpha^n=1$.  We say that
$\alpha\in\F$ is a \emph{primitive $n$\ts{th} root of unity}
if $\alpha\neq 0$ and the order of $\alpha$ in
 $\F^{\times}$ is $n$.
\end{definition}

\begin{fact.}[{\cite[Thm. 2.47(ii)]{lidl}}]
A primitive $n$\ts{th} root of unity exists in the finite
field of order $q$ if and only if $n\mid q-1$.
\end{fact.}

\begin{fact.}   \label{sumofroots}
Let $\F$ be a finite field of order $q$ and let $n\mid q-1$.
Then the number of $n$\ts{th} roots of unity is precisely $n$,
they are the powers of any primitive $n$\ts{th} root of unity,
and their sum is $0$ if $n\ge 2$ and $1$ if $n=1$.
\end{fact.}
\begin{proof}
The $n$\ts{th} roots of unity are precisely the roots of
the polynomial $x^n-1$.  They form a multiplicative group
which is therefore cyclic.  The primitive $n$\ts{th} roots
of unity in $\F$ are precisely the generators of this group.
The sum of the roots of $x^n-1$ is the negative of the
coefficient of $x^{n-1}$ in $x^n-1$.
\end{proof}
 
\begin{fact.}\label{order}
Let $\F_q$ be the finite field of order $q$ 
and let $n$ be an integer with $\gcd(q,n) = 1$.  
Let $\omega$ be a primitive $n$\ts{th} root of unity
in some extension field of $\F_q$.
Then the degree of the minimal polynomial of $\omega$ over
$\F_q$ is $\ord_q(n)$, the order of $q$ modulo $n$.
The conjugates of $\omega$ are
\[
\omega, \omega^q, \dots , \omega^{q^{\ord_q(n)-1}} \,.
\]
\end{fact.}
\begin{proof}
Immediate from \expref{Fact}{order1}.
\end{proof}

\subsection{Modifications to the main proofs}
In this section, we sketch how to modify the main proofs to work over finite fields.  We work over a finite field $\F_q$ where $q$ is a fixed constant (when we say parameters are chosen to be sufficiently large, they may be chosen in terms of $q$).  We will first define the $\DFT$ matrices over finite fields.
\begin{definition}
Let $\F_q$ be a finite field and $N$ be an integer with $\gcd(N,q) = 1$.  Pick a canonical primitive $N$\ts{th} root of unity $\omega$ in some extension of $\F_q$.  The $(x,y)$ entry of the $N \times N$ matrix $\DFT_{N,\F_q}$ ($0 \leq x,y \leq N - 1$) is $\omega^{xy}$.  We will omit the second subscript $\F_q$ and just write $\DFT_N$ when the base field is clear from context.
\end{definition}
\begin{remark}\label{finitefield_basicprops}
The matrix $\DFT_{N, \F_q}$ is well-defined over any extension of $\F_q$ that contains $\omega$.  It can easily be verified that the properties proved in \expref{Section}{sec_prelim} (namely that $\DFT_N$ diagonalizes circulant matrices) also hold in the finite field setting.
\end{remark}

The first lemma in this section allows us to lift to a field extension and then argue that if a matrix is highly non-rigid over some low-degree extension then it also cannot be rigid over the base field.

\begin{lemma}\label{conjugates}
Consider a finite field $\F_q$ and
a finite extension
$\F_q[\gamma]$ where $\gamma \neq 0$.  If the degree of
$\gamma$ over $\F_q$ 
is $g$ then for any matrix $M \in \F_{q}^{n \times n}$
and any positive integer $r$,
\[
\textsf{r}^{\F_q}_M(gr) \leq \textsf{r}^{\F_q[\gamma]}_M(r) \,.
\]
\end{lemma}
\begin{proof}
  Let the conjugates of $\gamma$ be $\gamma_1, \dots , \gamma_g$ where $\gamma_1 = \gamma$.  Let $k$ be a positive integer such that $\gamma_1^k + \dots + \gamma_g^k \neq 0$.
Such $0\le k\le g-1$ exists because the columns of the Vandermonde matrix
generated by the $\gamma_i$ are linearly independent. %
Let $s = \textsf{r}^{\F_q[\gamma]}_M(r)$.  There must be a matrix $E \in \F_q[\gamma]^{n \times n}$ with at most $s$ nonzero entries in each row and column such that $\rank_{\F_q[\gamma]}(M - E) \leq r$.  Now consider the $g$ matrices $E_1 = E, E_2, \dots , E_g$ where $E_i$ is obtained by taking $E$ and replacing $\gamma$ with its $i$\ts{th} conjugate, $\gamma_i$.  While naturally, we would like to consider the matrix $(E_1 + \dots + E_g)/g$ and write
\[
M  - \frac{E_1 + \dots + E_g}{g} = \frac{M - E_1}{g} + \dots + \frac{M - E_g}{g} \,,
\]
the above expression is only valid when $\gcd(g,q) = 1$ so we will need a slight modification.  Define the matrix $E'$ as follows.
\[
E' = \frac{1}{\gamma_1^k + \dots + \gamma_g^k} \cdot
\left(\gamma_1^k E_1 + \dots + \gamma_g^kE_g \right)\,.
\]

Note that $\gamma_1^k + \dots + \gamma_g^k \in \F_q$ and also $\gamma_1^kE_1 + \dots + \gamma_g^kE_g \in \F_{q}^{n \times n}$ since the entries are symmetric polynomials in $(\gamma_1, \dots , \gamma_g)$.  Thus $E' \in \F_{q}^{n \times n}$ and $E'$ clearly has at most $s$ nonzero entries in each row and column.  Next we observe that %
\[
M -  E'= \frac{1}{\gamma_1^k + \dots + \gamma_g^k}\cdot
 \left( \gamma_1^k(M - E_1) + \dots + \gamma_g^k(M - E_g)\right) \,.
\]
Note that $\rank_{\F_q[\gamma_i]}(M - E_i) \leq r$ for all $i$.  This is because the determinant of every $r \times r$ submatrix of $M - E$ can be written as a formal polynomial in $\gamma$ with coefficients in $\F_q$ and since $\gamma$ is a root of each of these polynomials, $\gamma_i$ must be as well, implying that the determinant of each of the $r \times r$ submatrices of $M - E_i$ is $0$.  Thus, we conclude that $\rank_{\F_q}(M - E') \leq gr$.  Writing $M = (M - E') + E'$, we immediately get the desired conclusion.
\end{proof}

Following the proof of \expref{Theorem}{Hadamard}, we can prove an analogue over finite fields.  All we needed in \expref{Theorem}{Hadamard} was that we were working over a field that contained the roots of unity in the definition of the GWH matrix.  Over finite fields, it suffices to work over an extension that contains the necessary roots of unity. 

\begin{theorem} \label{Hadamard_finitefield}
Let $\F_q$ be a finite field and $N = d^n$ for positive integers $d,n,q$ with $\gcd(d,q) = 1$.  Let $\omega$ be a primitive $d$\ts{th} root of unity in some extension of $\F_q$.  Let $0 < \epsilon < 0.01 $ and assume $n \geq 1/\psi$ where \[\psi = \frac{\epsilon^2 }{400\log^2 (1/\epsilon) d \log d} \,.\]  Let $H_{d,n} = \underbrace{\DFT_d \otimes \cdots \otimes \DFT_d}_{n}$. Then
\[
\textsf{r}_{H_{d,n}}^{\F_q[\omega]}\left(N^{1 - \psi}\right) \leq N^{ \epsilon} \,.
\]
\end{theorem}

With the above, we can now prove a finite field version of \expref{Lemma}{productbound}.  The proof is the same as the proof of \expref{Lemma}{productbound}, using \expref{Theorem}{Hadamard_finitefield} in place of \expref{Theorem}{Hadamard}.  The only necessary change is that we need to work over an extension of $\F_q$ that contains all of the necessary roots of unity. 
\begin{lemma}\label{finitefield_productbound}
Let $0 < \epsilon < 0.01$ be some chosen parameter, $\F_q$ be a fixed finite field, and $D$ be some sufficiently large constant (possibly depending on $\epsilon$ and $q$).  Consider positive integers $t_1 \leq t_2 \dots \leq t_n$ with $\gcd(t_i,q) = 1$ for all $i$.  Also assume $a_i \geq \max\left(\frac{t_i^2(\log t_i)^2}{\epsilon^{10}}, D \right)$ for all $i$.  Let $P = t_1^{a_1} \cdots t_n^{a_n}$ and $L = \lceil2 \log \log P \rceil$.  Consider the field extension $\F_q[\omega_1, \dots , \omega_n]$ where $\omega_i$ is a primitive $t_i$\ts{th} root of unity.\footnote{The $\omega_i$ need not be distinct.  We only need $\F_q[\omega_1, \dots , \omega_n]$ to be an extension that contains all of $\omega_1, \dots , \omega_n$.}  Let %
$\DFT_{t_i}$ %
be the $t_i \times t_i$ %
DFT %
matrix with entries %
in %
the field extension.  Let
 \[
A = (\underbrace{\DFT_{t_1} \otimes \cdots \otimes \DFT_{t_1}}_\text{$a_1$}) \otimes \cdots \otimes  (\underbrace{\DFT_{t_n} \otimes \cdots \otimes \DFT_{t_n}}_\text{$a_n$}) \,.
\]
Then we have
\[
\textsf{r}^{\F_q[\omega_1, \dots , \omega_n]}_A\left( P^{1 - \epsilon^6/(10L\,t_n^2 \log t_n)}\right) \leq  P^{5 \epsilon} \,.
\]
\end{lemma}

We also need a slight modification in the proof of \expref{Lemma}{blockrigidity}.  We will use the following definitions from \expref{Section}{sec_Fourier1}.
\begin{itemize}
\item $x$ is a sufficiently large integer.
\item $l$ is an integer such that 
\[
g_{100}(x) \leq l \leq g_{10}(x)
\]
where $g_k(x)$ is defined as in \expref{Notation}{not:bypolylog}.
\item $N$ is $(l,x)$-factorable.
\item $\mult_N(S), \fact_N(S)$ are defined as in \expref{Definition}{def_multandfact} and the matrix $M(S)$ is defined as in \expref{Definition}{def_factormatrix}.
\end{itemize}
\begin{remark}
Note that since $x$ is sufficiently large and $N$ is $(l,x)$-factorable, $\gcd(N,q) = 1$.  This will be important later on.  
\end{remark}
\begin{remark}
Note that if $\gamma$ is a primitive $N$\ts{th} root of unity, the matrix $M(S)$ is defined over the extension $\F_q[\gamma]$ for all subsets $S$.
\end{remark}

\begin{lemma}\label{lem_finitefieldblocks}
Let $\F_q$ be a fixed finite field.  Let $x$ be sufficiently large and $N = q_1q_2\cdots q_l$ be an $(l, x)$-factorable number with $\gcd(N,q) = 1$ and $g_{100}(x) \leq l \leq g_{10}(x)$. Let $t_1, \dots , t_a$ be the set of prime powers at most $x^{0.3}$ that are relatively prime to $q$.  Let $\omega_1, \dots , \omega_a$ be primitive $t_1\ts{th}, \dots , t_a\ts{th}$ 
roots of unity and let $\gamma$ be a primitive $N$\ts{th} root of unity.  For a subset $S \subset [l]$ with $|S| = k$ and $M(S)$ (as defined in \expref{Definition}{def_factormatrix})  a  $\fact_N(S) \times \fact_N(S)$ matrix,  we have 
\[
\textsf{r}^{\F_q[\gamma, \omega_1, \dots , \omega_a]}_{M(S)}\left( \frac{\fact_N(S)}{\exp\left(\epsilon^6 x^{0.37}\right)} \right) \leq \left(\fact_N(S)\right)^{6 \epsilon}
\]
as long as $k \geq \frac{x}{(\log x)^{C_0 + 200}}$.
\end{lemma}
\begin{proof}[Proof Sketch]
Without loss of generality $S = \{1,2, \dots , k \}$.  Recall that in the proof of \expref{Lemma}{blockrigidity}, we argued that the matrix $M(S)$ is $\Z_{q_1-1} \times \cdots \times \Z_{q_k - 1}$ circulant.  Then, using the prime factorizations of each of $q_1 -1, \dots , q_k - 1$, we wrote $\Z_{q_1-1} \times \cdots \times \Z_{q_k - 1}$ as a direct product of cyclic groups of prime power order.  Since each $q_i - 1$ must factor into prime powers that are at most $x^{0.3}$, we argued that some of these cyclic groups must appear many times in the direct product and then we could apply \expref{Lemma}{productbound}.  Over finite fields, the only necessary change in the proof of \expref{Lemma}{blockrigidity} is due to the fact that for an integer $b$ with $\gcd(b,q) > 1$, primitive $b$\ts{th} roots of unity do not exist over an extension of $\F_q$.  Thus, in the direct product of cyclic groups of prime power order, we cannot use those factors whose order is not relatively prime to $q$ (because we cannot diagonalize $\Z_b$-circulant matrices when $\gcd(b,q) > 1$).  To deal with this, we will use a more precise bound than (\ref{separateprimepowers}) where prime powers not relatively prime to $q$ are also excluded from the product on the left hand side.

Consider the factorizations of $q_1-1, \dots , q_k-1$ into prime powers.  For each prime power $p_i^{e_i}$ with $p_i^{e_i} \leq x^{0.3}$, let $c(p_i^{e_i})$ be the number of indices $j$ for which $p_i^{e_i}$ appears (exactly) in the factorization of $q_j-1$.  Also let $p$ be the characteristic of the finite field $\F_q$ that we are working over (so $q$ is a power of $p$).  Note that
 \[
 (q_1-1) \cdots (q_k-1) = \prod_{t}t^{c(t)} = p^{c(p)}p^{2c(p^2)}
 \cdots p^{fc(p^f)}\prod_{gcd(t,p) = 1}t^{c(t)} \,,
 \]
 where $t$ ranges over all prime powers at most $x^{0.3}$ and $p^f$ is the largest power of $p$ that is at most $x^{0.3}$.  For a power of $p$, say $p^i$, let $d(p^i)$ be the number of indices $j$ such that $q_j - 1$ is divisible (not necessarily exactly divisible) by $p^i$.  Let $L = \lfloor (1000 + C_0) \log_p \log x \rfloor $.  Then we have
 \begin{align*}
 p^{c(p)}p^{2c(p^2)} \cdots p^{fc(p^f)} &= p^{d(p) + d(p^2) + \dots + d(p^f)} \leq p^{\sum_{i =1}^L d(p^i) + \sum_{i= L+1}^f d(p^i)}  \leq p^{Lk + f x/(\log x)^{1000 + C_0}} \\ &\leq (\log x)^{(1000 + C_0)k}x^{x/(\log x)^{1000 + C_0}} \,.
 \end{align*}
 Next, consider all prime powers $p_i^{e_i}$ for which $c(p_i^{e_i}) < x^{0.62}$.  These satisfy
 \[
 \prod_{t, c(t) \leq x^{0.62}} t^{c(t)} \leq \left((x^{0.3})^{x^{0.62}} \right)^{x^{0.3}} \leq x^{x^{0.92}} \,.
 \]
Now without loss of generality, say $\{ t_1, \dots , t_n \}$ is the subset of $\{t_1, \dots , t_a \}$ ($n \leq a$) consisting of the set of prime powers for which $\gcd(t_i, p) = 1$ and $c(t_i) \geq x^{0.62}$.  
Let $P = t_1^{c(t_1)} \cdots t_n^{c(t_n)}$.  
{}From the above we know that as long as $x$ is sufficiently large
\[
P \geq \frac{\fact_N(S)}{x^{x^{0.92}} (\log x)^{(1000 + C_0)k}x^{g_{1000}(x)}} \geq \left(\fact_N(S)\right)^{(1 - \epsilon)}\cdot\frac{\left(\frac{x}{(\log x)^{C_0 + 1}} \right)^{\epsilon k}}{x^{x^{0.92}} (\log x)^{(1000 + C_0)k}x^{g_{1000}(x)}} \geq \left(\fact_N(S)\right)^{(1 - \epsilon)} \,.
\]
Recall $g_{1000}(x) = x/(\log x)^{1000 + C_0}$ is as defined in \expref{Notation}{not:bypolylog}.
The remainder of the proof can be completed in the same way as \expref{Lemma}{blockrigidity} using \expref{Lemma}{finitefield_productbound} in place of \expref{Lemma}{productbound}.
\end{proof}

Using the above we can prove the following analogue of \expref{Theorem}{main}.

\begin{theorem}\label{factorable_finitefield}
Let $\F_q$ be a fixed finite field and $0 <\epsilon < 0.01$ be some constant.  Let $x$ be sufficiently large and
$N = q_1q_2\cdots q_l$ be an
$(l, x)$-factorable number with $\gcd(N,q) = 1$ and $g_{100}(x) \leq l \leq g_{10}(x)$. Let $t_1, \dots , t_a$ be the set of prime powers at most $x^{0.3}$ that are relatively prime to $q$.  Let $\omega_1, \dots , \omega_a$ be primitive
$t_1\ts{th}, \dots , t_a\ts{th}$ 
roots of unity and let $\gamma$ be a primitive $N$\ts{th} root of unity.  Then 
\[
\textsf{r}^{\F_q[\gamma, \omega_1, \dots , \omega_n]}_{\DFT_N}\left( \frac{N}{\exp\left(\epsilon^6 (\log N)^{0.36}\right)} \right) \leq N^{ 7 \epsilon} \,.
\]
\end{theorem}
\begin{proof}
As in the proof of \expref{Theorem}{main}, we can subdivide the matrix $\DFT_N$ into submatrices of the form $M(S)$ for various subsets $S \subset [l]$ using \expref{Lemma}{division} (it is easily verified that \expref{Lemma}{division} also holds over finite fields).  We can remove all of the rows and columns corresponding to integers divisible by too many of the primes $q_1, \dots , q_l$ because the contribution of these rows and columns is low-rank.  The remaining entries can be subdivided into matrices of the form $M(S)$ where $|S|$ is sufficiently large so we can then apply \expref{Lemma}{lem_finitefieldblocks} to change a small number of entries in each row and column to reduce the rank significantly.  The precise computations are exactly the same as in \expref{Theorem}{main}.
\end{proof}

We will now combine \expref{Theorem}{factorable_finitefield} with \expref{Lemma}{conjugates} to get our main theorem for circulant matrices over finite fields.  
\begin{theorem}\label{circulant_finitefield}
Let $0 < \epsilon < 0.01$ be a given parameter and $\F_q$ be a fixed finite field.  For all sufficiently large $N$, if $M$ is an $N \times N$ %
 circulant or Toeplitz matrix %
 then   %
\[
\textsf{r}_{M}^{\F_q}\left(\frac{N}{\exp\left(\epsilon^6(\log N)^{0.35}\right)}\right) \leq N^{15 \epsilon} \,.
\]
\end{theorem}
\begin{proof}
First we analyze circulant matrices of size $N_0$ where $N_0$ is $(l,x)$-factorable for some 
\[
g_{100}(x) \leq l \leq g_{10}(x) \,.
\]
Note that as long as $x$ is sufficiently large, $N_0$ must be relatively prime to $q$.  Since $\DFT$ matrices  diagonalize circulant matrices (even over finite fields), \expref{Theorem}{factorable_finitefield} and \expref{Lemma}{diagonalization} imply that for $M_0$, an $N_0 \times N_0$ circulant matrix where $N_0$ satisfies the previously mentioned 
conditions, %
\[
\textsf{r}_{M_0}^{\F_q[\gamma, \omega_1, \dots , \omega_a]}\left(\frac{2N_0}{\exp\left(\epsilon^6(\log N_0)^{0.36}\right)}\right) \leq N_0^{14 \epsilon} 
\]
where $\gamma$ is a primitive $N$\ts{th} root of unity and $\omega_1, \dots , \omega_a$ are primitive $t_1\ts{th}, \dots , t_a\ts{th}$ roots of unity for $t_1, \dots , t_a$ being the set of prime powers at most $x^{0.3}$ that are relatively prime to $q$.  Now we analyze the degree of the extension $\F_q[\gamma, \omega_1, \dots , \omega_a]$.  Note $\F_q[\gamma, \omega_1, \dots , \omega_a] \subset \F_q[\eta]$ where $\eta$ is a primitive root of unity of order $C = N_0\lcm(t_1,t_2,\dots t_a)$.  By 
\expref{Fact}{order},  %
the degree of the extension $\F_q[\eta]$ is the order of $q$ modulo $C$. %
Since $N_0$ factors into a product of distinct $x$-good primes, %
by Fermat's little theorem, the order of $q$ modulo $N_0$ divides $\left(x^{0.3}\right)!$.  Also, all prime powers dividing $\lcm(t_1,t_2,\dots, t_a)$ are at most $x^{0.3}$.  Thus, the order of $q$ modulo $C$ divides $\left(x^{0.3}\right)!$.  Overall, the order of $q \mod C$ is at most
\[
\left(x^{0.3}\right)! < x^{0.3x^{0.3}} \leq \exp\left((\log N_0)^{0.31}\right) \,.
\]
Thus the degree of the extension $\F_q[\gamma, \omega_1, \dots , \omega_a]$ is at most $\exp\left((\log N_0)^{0.31}\right)$.  By \expref{Lemma}{conjugates}
\[
\textsf{r}_{M_0}^{\F_q}\left(\frac{N_0}{\exp\left(\epsilon^6(\log N_0)^{0.359}\right)}\right) \leq N_0^{14 \epsilon} \,.
\]
To complete the proof, we can simply repeat the arguments in the proof of \expref{Theorem}{main_circulant}.  For a circulant matrix $M$ of arbitrary size $N \times N$, note that it is possible to embed an $M$ in the upper left corner of a circulant matrix of any size at least $2N$.  By \expref{Lemma}{scales}, there exists an $N_0$ that is $(l,x)$-factorable for some $g_{100}(x) \leq l \leq g_{10}(x)$ such that
\[
\frac{N_0}{(\log N_0)^2} \leq N \leq \frac{N_0}{2} \,.
\]
We deduce
\[
\textsf{r}_{M}^{\F_q}\left(\frac{N_0}{\exp\left(\epsilon^6(\log N_0)^{0.359}\right)}\right) \leq N_0^{14 \epsilon} \,.
\]
Rewriting the bounds in terms of $N$ we get
\[
\textsf{r}_{M}^{\F_q}\left(\frac{N}{\exp\left(\epsilon^6(\log N)^{0.35}\right)}\right) \leq N^{15 \epsilon} \,. \qedhere
\]
\end{proof}

\section{$G$-circulant matrices over finite fields} %
\label{sec_groupalgebra_finitefield}
We will now generalize \expref{Theorem}{allabelian} to matrices over a finite field $\F_q$ except we will require the additional condition that $\gcd(|G|,q) = 1$.  Write the underlying abelian group $G$ as a direct product of cyclic groups $\Z_{n_1} \times \cdots \times \Z_{n_a}$.  While for matrices with entries in $\C$, it sufficed to work with the Kronecker product of the DFT 
matrices $\DFT_{n_1} \otimes \cdots \otimes \DFT_{n_a}$, 
we require slightly different techniques for rigidity over a fixed finite field as an extension containing all of the necessary roots of unity could have too high degree.  Instead of working through DFT 
matrices, we will work directly with the $G$-circulant matrices themselves.  

While for sufficiently large cyclic groups, we did not require the condition that $\gcd(|G|,q) = 1$ (see \expref{Theorem}{circulant_finitefield}), we require the condition for general abelian groups because we need to use 
\expref{Theorem}{Hadamard_finitefield} to deal with the case when $G$ contains the direct product of many copies of a small cyclic group.  In particular, our techniques do not handle a group such as $\Z_{p^2} \times \cdots \times \Z_{p^2}$ where $p$ is equal to the characteristic of the field $\F_q$.  It is an interesting open question to see if the condition that $\gcd(|G|,q) = 1$ can be eliminated.  The work in \cite{CLP} deals with the case where $q = p^a$ for a prime $p$ and $G$ is a direct product of many cyclic groups of order $p$ but not the case when $G$ is a direct product of many cyclic groups of order $p^2$ (or some other power of $p$).

The first important observation is that \expref{Theorem}{main_circulant} can be slightly strengthened so that to reduce the rank of any circulant matrices, the locations to be changed are fixed and the changes are fixed linear combinations of the entries of the circulant matrix.  More precisely, we make the following definition.
\begin{define}\label{group_reducibility}
Given a group $G$   %
of order   %
$|G| = n$, we say $G$ is $(r,s)$-reducible over
$\F_q$ if the following %
condition holds.   %
There exist
\begin{itemize}
\item a set $S \subset [n] \times [n]$ of positions where $S$
  contains at most $s$
  positions in each row and column,
\item  matrices $A,B \in \F_q^{n\times n}$ where
   $\rank(A), \rank(B) \leq r$,
\item matrices $E_1, \dots , E_n \in \F_q^{n \times n}$ with
  all nonzero entries in $S$, and
\item matrices $Y_1, \dots , Y_n, Z_1, \dots , Z_n \in \F_q^{n \times n}$ 
\end{itemize}
such that for any $G$-circulant matrix $M$ with top row
$(x_1, \dots , x_n)$, we have  %
    \[
    M = A(x_1Y_1 + \dots + x_nY_n) + (x_1Z_1 + \dots + x_nZ_n)B + (x_1E_1 + \dots + x_nE_n) \,.
    \]

In such a decomposition, the matrices $A,B$ will be called $(r,s)$-reduction matrices and the matrices
$Y_1, \dots , Y_n,Z_1, \dots , Z_n, E_1, \dots , E_n$ will be called
$(r,s)$-reduction helpers. %
We write $Y_M = x_1Y_1 + \dots + x_nY_n$ and
similarly  %
for $Z_M$ and $E_M$.
\end{define}

Following the proof of \expref{Theorem}{main_circulant}, 
we can show that $\Z_N$ is 
\[
\left(\frac{N}{\exp\left(\epsilon^6(\log N)^{0.35}\right)}\,,\,N^{15 \epsilon} \right)
\]
reducible over $\C$.  We now prove an analogue of this result for finite fields.
\begin{claim}\label{main_stronger}
For fixed $0 < \epsilon < 0.01$ and all sufficiently large $N$, the group $\Z_N$ is
\[
\left( \frac{N}{\exp\left(\epsilon^6(\log N)^{0.35}\right)},  N^{15 \epsilon}\right) 
\]
reducible over $\F_q$.
\end{claim}
\begin{proof}
  First consider an integer $N_0$ that is $(l,x)$-factorable for some $
g_{100}(x) \leq l \leq g_{10}(x)$.  As long as $x$ is sufficiently large, $\gcd(N_0,q) = 1$.  Let $M_0$ be a $N_0 \times N_0$ circulant matrix (i.e., a
$G$-circulant    %
matrix for
$G=\Z_{N_0}$)  %
over $\F_q$ and 
let the entries in its top row be %
$x_1, \dots , x_{N_0}$.  Let $\gamma$ be a primitive $N_0$\ts{th} root of unity and $t_1, \dots t_n$ be the set of prime powers at most $x^{0.3}$ that are relatively prime to $q$.  Let $\omega_1, \dots , \omega_n$ be roots of unity of order $t_1, \dots , t_n$ respectively.  By \expref{Theorem}{factorable_finitefield}, there exists a matrix $E$ over $\F_q[\gamma, \omega_1, \dots , \omega_n]$ with at most $N_0^{7\epsilon}$ nonzero entries in each row and column such that 
\[
\rank(\DFT_{N_0} - E) \leq \frac{N_0}{\exp\left(\epsilon^6(\log N_0)^{0.36}\right)} \,.
\]
Now write
\[
M_0 = \DFT_{N_0}^* \cdot D \cdot \DFT_{N_0} = (\DFT_{N_0} - E)^*  D \cdot \DFT_{N_0} + E^*D(\DFT_{N_0} - E) + E^*DE
\]
where $D$ is a diagonal matrix whose entries are linear combinations of $x_1,\dots , x_{N_0}$.  Note that all of the above matrices have entries contained in $\F_q[\gamma, \omega_1, \dots , \omega_n] \subseteq \F_q[\eta]$ where $\eta$ is a primitive root of unity of order $C = N_0\lcm(t_1, \dots , t_n)$.  As argued before in the proof of \expref{Theorem}{circulant_finitefield}, the degree of the extension is at most $\exp\left((\log N_0)^{0.31}\right)$.  Let the conjugates of $\eta$ be $\eta_1 = \eta, \eta_2, \dots , \eta_m$.  Let %
$\DFT_{N_0}^1, \dots , \DFT_{N_0}^m$ 
be obtained by taking %
$\DFT_{N_0}$ %
and replacing $\eta$ with its conjugates. Define $D^1, \dots , D^m, E^1, \dots , E^m$ similarly.  As in the proof of \expref{Lemma}{conjugates}, there exists an integer $k$ such that $\eta_1^k +  \dots + \eta_m^k \neq 0$.  We now have
\[
M_0 = \frac{1}{\eta_1^k + \dots + \eta_m^k}\left(\sum_{i=1}^m \eta_i^k(\DFT_{N_0}^i - E^i)^*D^i \cdot \DFT_{N_0}^i + \sum_{i=1}^m \eta_i^k{E^i}^*D^i(\DFT_{N_0}^i - E^i) + \sum_{i=1}^m \eta_i^k{E^i}^*D^iE^i\right) \,.
\]
Note that $1/(\eta_1^k + \dots + \eta_m^k) \in \F_q$ and all three of the sums are matrices whose entries are linear combinations of $x_1, \dots , x_{N_0}$ with coefficients in $\F_q$.  The last term satisfies the desired sparsity constraint as it has at most $N_0^{14 \epsilon}$ nonzero entries in each row and column and the locations of these entries are independent of $M_0$.  

It remains to argue that the first two terms satisfy the desired rank constraint.  Note that the span of the columns of 
$(\DFT_{N_0}^1 - E^1), \dots , (\DFT_{N_0}^m - E^m)$
has dimension at most 
\[
\frac{mN_0}{\exp\left(\epsilon^6(\log N_0)^{0.36}\right)} \leq \frac{N_0}{\exp\left(\epsilon^6(\log N_0)^{0.359}\right)}
\]
over $\F_q[\eta]^{N_0}$.  Therefore, the dimension of the intersection of this subspace with $\F_q^{N_0}$, say $V$, has dimension at most 
\[
\frac{N_0}{\exp\left(\epsilon^6(\log N_0)^{0.359}\right)} \,.
\]
In particular we can write
\[
\sum_{i=1}^m \eta_i^k(\DFT_{N_0}^i - E^i)^*D^i\cdot \DFT_{N_0}^i = x_1C_1  + \dots  + x_{N_0}C_{N_0}
\]
for some fixed matrices $C_1, \dots , C_{N_0}$ with entries %
in %
$\F_q$.  Also all columns of $C_1, \dots , C_{N_0}$ must be in $V$ so each can be written as $AY_i$ where $A$ is a fixed matrix with rank at most \[\frac{mN_0}{\exp\left(\epsilon^6(\log N_0)^{0.36}\right)} \,.\]  Thus there exists fixed matrices $Y_1, \dots , Y_{N_0} \in \F_q^{n \times n}$ and a matrix $A$ satisfying the desired rank constraint such that
\[
\sum_{i=1}^m \eta_i^k(\DFT_{N_0}^i - E^i)^*D^i \cdot \DFT_{N_0}^i = A(x_1Y_1 + \dots + x_{N_0}Y_{N_0}) \,.
\]
A similar argument shows that the second term can also be written in the desired form.

Now to extend to arbitrary $N$ (not necessarily $(l,x)$-factorable),
simply note that any circulant matrix of size $N$ can be embedded
into a circulant matrix of
any given  %
size at least $2N$ where
each entry of the larger matrix is equal to some entry %
of the original matrix.  
We can then apply \expref{Lemma}{scales} and complete the proof in the same way as \expref{Theorem}{main_circulant}.
\end{proof}

\expref{Claim}{main_stronger} allows us to deal with large cyclic groups.  We will also need a way of dealing with a direct product of many copies of a small cyclic group.  
\begin{claim}\label{smallgroups}
Let $\F_q$ be a finite field and $N = d^n$ for positive integers $d,n,q$ with $\gcd(d,q) = 1$.  Let $0 < \epsilon < 0.01 $ and assume $n \geq 2/\psi$ where \[\psi = \frac{\epsilon^2 }{400\log^2 (1/\epsilon) d \log d} \,.\]  Let $G = \underbrace{\Z_d \otimes \cdots \otimes \Z_d}_{n}$. Then $G$ is $\left(N^{1 - \psi/2},  N^{ 2\epsilon} \right)$ reducible over $\F_q$.
\end{claim}
\begin{proof}
Let $\omega$ be a primitive $d$\ts{th} root of unity in some extension of $\F_q$.  We use \expref{Theorem}{Hadamard_finitefield} to find a sparse matrix $E$ with entries in  $\F_q[\omega]$ such that $H_{d,n} - E$ has low rank where
\[
H_{d,n} = \underbrace{\DFT_d \otimes \cdots \otimes \DFT_d}_{n} \,.
\]
We can then repeat the same argument as in the proof of \expref{Claim}{main_stronger}, using the fact that $H_{d,n}$ diagonalizes any $G$-circulant matrix.
\end{proof}

Now we introduce the main technical result of this section that allows us to deal with direct products of different groups without going through the corresponding $\DFT$ matrices.
\begin{claim}\label{productrigidity}
Consider a list of
abelian  %
groups, $G_1, \dots , G_a$, such that $|G_i| = n_i$.  
Assume  
for each $1 \leq i \leq a$\,, $G_i$ is $(r_i,s_i)$-reducible over $\F_q$.
Let %
$G = G_1 \times \cdots \times G_a$   %
and $|G| = n = n_1n_2 \cdots n_a$.  Then the group $G$ is $(r,s)$-reducible over $\F_q$ where 
\begin{align*}
    r = \sum_{S \subset [a], |S| = l} 2^l\prod_{i \in S} \sqrt{r_i n_i} \prod_{i' \notin S} n_{i'}\,, \\
    s = \sum_{S \subset[a], |S| < l} 2^{|S|}\prod_{i \in S} n_i \prod_{i' \notin S} s_{i'} \,.
\end{align*}
\end{claim}
\begin{proof}
Let $M$ be a $G$-circulant matrix.  For each $1 \leq i \leq a$, let $A^i,B^i$ be the $(r_i,s_i)$-reduction matrices for the group $G_i$.  Note that the reducibility assumption means that any $G_i$-circulant matrix can be written as a sum of three matrices where the first contains a fixed high dimensional subspace in its left nullspace, the second contains a fixed high dimensional subspace in its right nullspace, and the third is sparse.  The first step in our proof will involve writing $M$ as a sum of $3^a$ matrices.  Roughly, each of these $3^a$ matrices corresponds to choosing one of the three possible components (large left nullspace, large right nullspace, or sparse) for each of the groups $G_i$.  

More formally, for each $i \in [a]$, consider the group $G_i$. Let its $(r_i,s_i)$-reduction helpers be $\{Y_{g_i} \}$, $\{Z_{g_i} \}$, $\{E_{g_i}\}$ for $g_i \in G_i$.   Let $Y(g_i) = A^{i}Y_{g_i}$, $Z(g_i) = Z_{g_i}B^i$ and  $E(g_i) = E_{g_i}$.  By definition, for a $G$-circulant matrix $M_{G_i}$ with top row given by $\{x_{g} \}$ for  $g \in G_i$,
\[
M_{G_i} = \sum_{g_i \in G} (Y(g_i) + Z(g_i) + E(g_i))x_{g_i} \,.
\]
Thus, for fixed $g_i, h_i,k_i \in G_i$, the entry of $Y( g_i) + Z( g_i) + E(g_i)$ indexed by $(h_i,k_i)$ is equal to  $1$ if $h_i + k_i = g_i$ (in the group $G_i$) and $0$ otherwise.

We index the rows and columns of $M$ with ordered tuples $(h_1, \dots , h_a)$ and $(k_1, \dots , k_a)$ respectively (where $h_i,k_i \in G_i$).  Let the entries in the top row of $M$ be $x_{g_1, \dots , g_a}$ where $(g_1, \dots , g_a)$ ranges over  $ G_1 \times \cdots \times G_a$. Now for each ordered tuple $I = (i_1, \dots ,i_a) \in \{1,2,3 \}^a$, we will construct a $|G| \times |G|$ matrix $M_I$.  Let $S^1(I), S^2(I), S^3(I) \subset [a]$ denote the subsets of locations where the entry of $I$ is $1,2$ or $3$ respectively.  We define
\[
M_I = \sum_{(g_1, \dots , g_a)} x_{g_1, \dots , g_a} \left(\bigotimes_{i \in S^1(I)}Y( g_i) \right) \otimes \left(\bigotimes_{i \in S^2(I)}Z( g_i)\right) \otimes \left(\bigotimes_{i \in S^3(I)}E(g_i) \right)
\] 
where the sum is over all $(g_1, \dots , g_a) \in  G_1 \times \cdots \times G_a$.  The first important observation is that 
\begin{equation}\label{eq_sum}
M =  \sum_{I \in \{1,2,3 \}^a} M_I \,.
\end{equation}
To see this, it suffices to note that the coefficients of $x_{g_1, \dots , g_a}$ on the right hand side for fixed $g_1, \dots , g_a$ are given by the matrix 
\begin{align*}
\sum_{I \in \{1,2,3 \}^a}  \left(\bigotimes_{i \in S^1(I)}Y( g_i) \right) \otimes \left(\bigotimes_{i \in S^2(I)}Z( g_i)\right) \otimes \left(\bigotimes_{i \in S^3(I)}E(g_i) \right) = \\
\bigotimes_{i \in [a]}\left(Y( g_i) + Z(g_i) + E(g_i) \right) \,.
\end{align*}
The entry indexed by $(h_1, \dots , h_a)$ and $(k_1, \dots , k_a)$ on the right hand side is equal to $1$ if $h_i + k_i = g_i$ for all $i$ and $0$ otherwise.  This completes the proof of \eqref{eq_sum}.

We would like to write $M $ as a sum of three matrices, say $P_1, P_2, P_3$, whose entries are linear forms in the variables $x_{g_1, \dots , g_a}$ and such that $P_1 = AY, P_2 = ZB$ for some fixed low-rank matrices $A,B$ and $P_3$ is sparse.  Write
\[
M = \sum_{\substack{I \in \{1,2,3 \}^a \\ |S^3(I)| \leq  a - l}}M_{I}  + \sum_{\substack{I \in \{1,2,3 \}^a \\ |S^3(I)| > a - l}}M_{I} \,.
\]
We will prove that $P_1,P_2$ can be obtained by splitting the first sum and we can set $P_3$ to be equal to the second sum.  For each $1 \leq i \leq a$, there exists a set of linearly independent vectors $v_1^i, \dots , v_{n_i - r_i}^i$ such that $v_{j}^iA^i = 0$ and a set of linearly independent vectors $u_1^i, \dots , u_{n_i - r_i}^i$ such that $B^iu_j^i = 0$ for all $1 \leq j \leq n_i - r_i$.  We can complete the set  $\{ v_1^i, \dots , v_{n_i - r_i}^i \}$ to a basis $\{ v_1^i, \dots , v_{n_i }^i \}$ and similar for $\{u_1^i, \dots , u_{n_i}^i \}$.  Consider the basis of $\F_q^{n}$ consisting of the vectors $v^1_{j_1} \otimes v^2_{j_2} \otimes \cdots \otimes v^a_{j_a}$ where $(j_1, \dots, j_a) \in [n_1] \times \cdots \times [n_a]$.  Now 
assume %
we are given a matrix $M_{I}$ with $I \in \{1,2,3 \}^a$.    The key observation is that if %
$j_{i} \leq n_{i} - r_{i}$ for some index $i \in S^1(I)$, %
then 
\begin{equation}\label{eq_null}
\left(v^1_{j_1} \otimes v^2_{j_2} \otimes \cdots \otimes v^a_{j_a}\right)M_I = 0\,.
\end{equation}
This is because $Y(g_i) = A^iY_{g_i}$ so by construction, $v^i_{j_i} Y(g_i) = 0$ for all $g_i \in G_i$.  Using this observation and examining the definition of $M_I$, we immediately get \eqref{eq_null}.  Similarly, we get that if $j_{i} \leq n_{i} - r_{i}$ for some index $i \in S^2(I)$ then 
\[
M_I\left(u^1_{j_1} \otimes u^2_{j_2} \otimes \cdots \otimes u^a_{j_a}\right) = 0 \,.
\]
Let $R_1 \subset \{1,2,3 \}^a$ be the set of ordered tuples $I$ such that 
\[
\prod_{i \in S^1(I)} \frac{r_i}{n_i} \leq \prod_{i \in S^2(I)} \frac{r_i}{n_i} \,.
\]
Let $R_2 = \{1,2,3 \}^a \backslash R_1$.  We now write
\[
\sum_{\substack{I \in \{1,2,3 \}^a \\ |S^3(I)| \leq  a - l}}M_{I} = \sum_{\substack{I \in R_1 \\ |S^3(I)| \leq  a - l}}M_{I} + \sum_{\substack{I \in R_2 \\ |S^3(I)| \leq  a - l}}M_{I}
\]
and will argue that the first term, which we call $P_1$, has a fixed, high dimensional subspace contained in its left nullspace while the second term, which we call $P_2$, has a fixed, high dimensional subspace contained in its right nullspace.  We work with the basis $v^1_{j_1} \otimes v^2_{j_2} \otimes \cdots \otimes v^a_{j_a}$ where $(j_1, \dots, j_a) \in [n_1] \times \cdots \times [n_a]$ and count the number of these basis vectors that are not in the left nullspace of $P_1$.  For each $I \in R_1$, by \eqref{eq_null}, $M_I$ contributes at most 
\[
\prod_{i \in S^1(I)}r_i \prod_{i \in [a] \backslash S^1(I)}n_i 
\]
basis vectors for which $vM_I$ is nonzero.  Furthermore if $S^1(I) \subset S^1(I')$ for two distinct ordered tuples $I$ and $I'$, the contributions of $M_{I'}$ are redundant with the contributions of $M_I$.  Thus, we can ignore the contributions of $M_I$ for ordered tuples $I$ for which $|S^3(I)| < a - l$.   Overall, the number of basis vectors outside the left nullspace of $P_1$ is at most 
\begin{equation}\label{eq-rankbound}
\sum_{\substack{I \in R_1 \\ |S^3(I)| = a - l}} \prod_{i \in S^1(I)}r_i \prod_{i \in [a] \backslash S^1(I)}n_i \leq \sum_{\substack{I \in \{1,2,3 \}^a \\ |S^3(I)| = a - l}} \prod_{i \in [a] \backslash S^3(I)}\sqrt{r_in_i}\prod_{i \in S^3(I)}n_i = \sum_{S \subset [a], |S| = l} 2^l\prod_{i \in S} \sqrt{r_i n_i} \prod_{i' \notin S} n_{i'}  
\end{equation}
where to obtain the above inequality, we first used the fact that $I \in R_1$ and then used that for a fixed set $S^3(I)$, there are $2^l$ possible ordered tuples $I$.  Note that the basis $v^1_{j_1} \otimes v^2_{j_2} \otimes \cdots \otimes v^a_{j_a}$ where $(j_1, \dots, j_a) \in [n_1] \times \cdots \times [n_a]$ is fixed (i.e., independent of the entries of $M$).  Thus we can write $P_1 = AX$ where the entries of $X$ are linear forms in the entries of $M$ and $A$ is a fixed matrix with rank bounded above by the expression in \eqref{eq-rankbound}.  A similar argument allows us to write $P_2 = YB$ for a fixed matrix $B$ with the desired rank.  

Now it remains to bound the sparsity of 
\[
\sum_{\substack{I \in \{1,2,3 \}^a \\ |S^3(I)| > a - l}}M_{I} \,.
\]
We claim that the number of nonzero entries in each row and column of $M_I$ is at most
\[
  \prod_{i \in S^3(I)}s_i \prod_{i \in [a] \backslash S^3(I)}n_i\,.
\]
To see this, note that for each $i$, the matrices $E(g_i)$, as $g_i$
ranges over all of $G_i$, have all of their nonzeros contained in a
fixed set $S_i$ where $S_i$ contains at most $s_i$ distinct locations
in each row and column.  For each fixed subset $ S^3(I)$, there are
exactly $2^{|S^3(I)|}$ possible ordered tuples $I$.  Thus the number
of nonzero entries in each row and column of the sum is at most
\[
\sum_{\substack{I \in \{1,2,3 \}^a \\ |S^3(I)| > a - l}}\prod_{i \in S^3(I)}s_i \prod_{i \in [a] \backslash S^3(I)}n_i = \sum_{S \subset [a],  |S| < l}2^{|S|}\prod_{i \in [a]\backslash S}s_i \prod_{i \in S}n_i \,.
\]
This completes the proof that the group $G$ is $(r,s)$-reducible over $\F_q$.
\end{proof}

We are now ready to prove the main theorem about rigidity of
$G$-circulant    %
matrices over finite fields.
\begin{theorem}\label{finitefield_allabelian}
Let $\F_q$ be a fixed finite field and $\epsilon < 0.01$ be a fixed constant.  Let $G$ be an abelian group.  As long as $|G|$ is sufficiently large and $\gcd(|G|,q) = 1$, for any 
$G$-circulant    %
matrix $M$
over $\F_q$, we have
\[
\textsf{r}_M^{\F_q}\left( \frac{|G|}{\exp\left(\epsilon^{20}(\log |G|)^{0.3}\right)}\right) \leq |G|^{100 \epsilon} \,.
\]
\end{theorem}
\begin{proof}
By the Fundamental Theorem of Finite Abelian Groups we can write $G = \Z_{n_1} \times \cdots \times \Z_{n_a}$.  The proof will essentially follow the same method as the proof of 
\expref{Theorem}{DFTabelian} 
except using \expref{Claim}{productrigidity} to deal with direct products
of cyclic groups that are roughly the same
order.   %

  Without loss of generality, $n_1 \leq n_2 \leq \dots \leq n_a$.  We will choose $k$ to be a fixed, sufficiently large positive integer (possibly depending on $q,\epsilon$).  Consider the ranges $I_1 = [k,k^2), I_2 = [k^2, k^4), \dots I_j = [k^{2^{j-1}},k^{2^j}) \dots $ and so on.  Let $S_j$ be a multiset defined by $S_j = I_j \cap \{n_1, \dots, n_a \}$.  Fix a $j$ and let the elements of $S_j$ be $x_1 \leq \dots \leq x_b$.  By \expref{Claim}{main_stronger}, we have that (since $k$ sufficiently large) for each $x_i$, the group $\Z_{x_i}$ is 
\begin{equation}\label{eq-reducibility}
\left( \frac{x_i}{\exp\left(\epsilon^6(\log x_i)^{0.35}\right)},  x_i^{15 \epsilon}\right) 
\end{equation}
reducible over $\F_q$.  Now we will use \expref{Claim}{productrigidity} to argue about the group $G_j = \Z_{x_1} \times \cdots \times \Z_{x_b}$.  Set $l = \lceil \epsilon b \rceil$ in \expref{Claim}{productrigidity}.  We get that $G_j = \Z_{x_1} \times \cdots \times \Z_{x_b}$ is $(r,s)$ reducible for some $r,s$ that are obtained by plugging \eqref{eq-reducibility} into the expressions in \expref{Claim}{productrigidity}.  We bound $r$ and $s$ more carefully below.  We have
\begin{align*}
r \leq \binom{b}{\lceil \epsilon b \rceil}\frac{2^{\lceil \epsilon b \rceil}x_1 \cdots x_b}{\left(\exp\left(\epsilon^6 (\log x_1)^{0.35}\right)\right)^{\lceil \epsilon b \rceil/2}} \leq \frac{(2b)^{\lceil \epsilon b \rceil}}{( \frac{\epsilon b}{3})^{\lceil \epsilon b \rceil}}\frac{x_1 \cdots x_b}{\left(\exp\left(\epsilon^6 (\log x_1)^{0.35}\right)\right)^{\lceil \epsilon b \rceil/2}} \\ = x_1 \cdots x_b\left( \frac{36}{\epsilon^2 \exp\left(\epsilon^6 (\log x_1)^{0.35}\right)} \right)^{\lceil \epsilon b \rceil/2} \,.
\end{align*}

As long as $k$ is sufficiently large, we have
\begin{eqnarray*}
r \leq x_1 \cdots x_b\left( \frac{36}{\epsilon^2 \exp\left(\epsilon^6 (\log x_1)^{0.35}\right)} \right)^{\lceil \epsilon b \rceil/2} \leq x_1 \cdots x_b\left( \frac{1}{\exp\left(\epsilon^6 (\log x_1)^{0.34}\right)} \right)^{\lceil \epsilon b \rceil/2}  \\ \leq \frac{ x_1 \cdots x_b}{\exp\left(\epsilon^7 (\log x_1 \cdots x_b )^{0.33}\right)}
\end{eqnarray*}
where in the last step we used the fact that $x_i \leq x_1^2$ for all $i$.  Next we bound the sparsity $s$.  We have
\[
s \leq 4^b x_b \cdots x_{b - \lfloor \epsilon b\rfloor + 1}(x_{b - \lfloor \epsilon b\rfloor} \cdots x_1)^{15\epsilon} = 4^b (x_1 \cdots x_b)^{15\epsilon}(x_b \cdots x_{b - \lfloor \epsilon b\rfloor + 1})^{1 - 15\epsilon} \leq (x_1 \cdots x_b)^{18\epsilon}
\]
where in the last step above, we used the fact that $x_i \leq x_1^2$ for all $i$.  Thus, we have shown that the group $G_j = \Z_{x_1} \times \cdots \times \Z_{x_b}$ is 
\[
\left( \frac{ x_1 \cdots x_b}{\exp\left(\epsilon^7 (\log x_1 \cdots x_b )^{0.33}\right)}, (x_1 \cdots x_b)^{18\epsilon}\right)
\]
reducible over $\F_q$.

The above allows us to deal with direct products of large cyclic groups that are all of roughly the same order.  In the direct product $G = \Z_{n_1} \times \cdots \times \Z_{n_a}$, we will split the terms into cyclic groups of small order, which can be dealt with using \expref{Claim}{smallgroups}, and several products of large cyclic groups that can each be dealt with using the above. We now formalize this argument.  For each integer $c$ between $2$ and $k$ with $\gcd(c,q) = 1$, let $m_c$ be the number of copies of $c$ in the set $\{n_1, \dots , n_a \}$.  If $m_c \geq k^2 (\log k)^2/\epsilon^4 $ then by \expref{Claim}{smallgroups}, the group $\underbrace{\Z_c \times \cdots \times \Z_c}_{m_c}$ is 
\[
\left( c^{m_c(1 - \epsilon^4/(k^2 (\log k)^2)}, c^{2m_c \epsilon} \right)
\]
reducible over $\F_q$.  Let $L = \lceil 2\log \log |G| \rceil$ and ensure that $|G|$ is sufficiently large so that $L > k$.  Let $T$ be the set of integers $c$ between $2$ and $k$ with $\gcd(c,q) = 1$ such that $c^{m_c} \geq |G|^{\epsilon/(2L)}$.  Note that as long as $|G|$ is sufficiently large, all elements of $T$ must satisfy $m_c \geq k^2 (\log k)^2/\epsilon^4 $.  Let $R$ be the set of indices $j$ for which $\prod_{x \in S_j}x \geq |G|^{\epsilon/(2L)}$.  Note that $S_j$ is clearly empty for $j \geq L$.  Recall that $\gcd(|G|,q) = 1$ so the group $G$ can be written as 

\[
G = \left( \bigtimes_{\substack{2 \leq c < k \\ \gcd(c,q) = 1}} \left( \underbrace{\Z_c \times \cdots \times \Z_c}_{\text{$m_c$}}\right) \right) \times \left(\bigtimes_{1 \leq j \leq L} G_j \right) \,.
\]

Define 

\[B = \left( \bigtimes_{c \notin T} \left( \underbrace{\Z_c \times \cdots \times \Z_c}_{\text{$m_c$}}\right) \right) \times \left(\bigotimes_{j \notin R} G_j \right) \,. \]

Note that 
\[
|B| \leq \left( |G|^{\epsilon/(2L)}\right)^{k + L} \leq |G|^{\epsilon} \,.
\]
Also $G = B \times D$ where

\[
D =  \left( \bigtimes_{c \in T} \left( \underbrace{\Z_c \otimes \cdots \otimes \Z_c}_{\text{$m_c$}}\right) \right) \times \left(\bigtimes_{j \in R} G_j \right) \,.
\]
Now we apply \expref{Claim}{productrigidity} again on $D$ where we view $D$ as a direct product of groups of the form $\underbrace{\Z_c \otimes \cdots \otimes \Z_c}_{\text{$m_c$}}$ and $G_j$ and we set $l = 1$.  We get that $D$ is $(r_D,s_D)$ reducible over $\F_q$ where
\begin{align*}
r_D &\leq 2|D|\left( \sum_{c \in T}\frac{1}{c^{m_c\epsilon^4/(2k^2 \log k)}} + \sum_{j \in R}\frac{1}{\exp\left(0.5\epsilon^7 (\log \prod_{x \in S_j}x )^{0.33}\right)} \right)  \leq \frac{|D|}{\exp\left(\epsilon^8 (\log |G|)^{0.32}\right)}\,, \\
s_D &\leq |D|^{18\eps} \,.
\end{align*}
Finally, note that the group $B$ is trivially $(0, |B|)$ reducible over $\F_q$.  Thus, by \expref{Claim}{productrigidity}, $G = B \times D$ is $(r_G,s_G)$ reducible over $\F_q$ for
\begin{align*}
r_G &\leq \frac{2|D| \cdot |B| }{\exp\left(0.5\epsilon^8 (\log |G|)^{0.32}\right)} \leq \frac{|G|}{\exp\left(\epsilon^{8}(\log |G|)^{0.31}\right)}\,, \\
s_G &\leq |B| \cdot |D|^{18\eps} \leq |G|^{19\eps} \,.
\end{align*}
The above immediately implies that for any $G$-circulant matrix $M$ with $\gcd(|G|,q) = 1$,
\[
\textsf{r}_M^{\F_q}\left( \frac{|G|}{\exp\left(\epsilon^{20}(\log |G|)^{0.3}\right)}\right) \leq |G|^{100 \epsilon} \,.
\]
This completes the proof. 

\end{proof}

\section{Final remarks and open questions}
\label{sec_conclusion} 
Our main results,
\expref{Theorems}{allabelian}, \ref{thm:extension_degree},
\expref{and}{finitefield_allabelian},
naturally raise
some     %
open questions.
Recall that the $N\times N$ DFT (Discrete Fourier Transform)
matrix is the matrix $(\omega^{ij})$ $(i,j=0,\dots, N-1)$
where $\omega$ is a primitive $N\ts{th}$ root of unity.
\begin{itemize}
\item Are the $N \times N$ DFT matrices rigid over
  the $N$\ts{th} cyclotomic field $\Q[\omega]$
  (where $\omega$ is a primitive $N$\ts{th} root of unity) ?
  (Compare this question with \expref{Theorem}{thm:extension_degree}.)
\item Do there exist circulant matrices, or $G$-circulant matrices for
   some class of abelian groups $G$, that are rigid over $\Q$\, ? 
  (Again, compare with \expref{Theorem}{thm:extension_degree}.)
\item Does there exist a finite field $\F_q$ and $G$-circulant matrices for
   some class of abelian groups $G$ with $\gcd(|G|, q ) > 1$ that are rigid over $\F_q$\, ? 
  (Compare this question with \expref{Theorem}{finitefield_allabelian}.)
\item Do there exist rigid $G$-circulant matrices over $\C$ 
  for some class of (necessarily non-abelian) groups $G$ ?
\end{itemize}

When $G$ is non-abelian, it is no longer possible to simultaneously 
diagonalize the matrices $M_G(f)$ for all $f$  
but there is a change of basis matrix $A$ such that $AM_G(f)A^*$
is block-diagonal where the diagonal blocks correspond to the irreducible
representations of $G$.  When all of the irreducible representations of $G$
have small degree (dimension),  
it may be possible to use similar techniques to the ones used here.
On the other hand, this suggests that perhaps $M_G(f)$ is a candidate
for rigidity when all irreducible representations of $G$
have large degree.  Frobenius proved in 1896 
that the group $SL_2(\F_p)$ of $2 \times 2$ matrices over $\F_p$ with
determinant $1$ has no nontrivial irreducible representations
of degree less than $(p-1)/2$ over $\C$ 
and thus has highly nonabelian structure \cite{frobenius1896}.
(See \cite{davidoff} for an accessible presentation.)  %
Thus we make the following conjecture. %
\begin{conjecture}
For large primes $p$, a random $G$-circulant $(0,1)$-matrix
$M_G(f)$ for $G = SL_2(\F_p)$ is Valiant-rigid over $\C$ with
high probability.  Here by ``random'' we mean the function
$f:SL_2(\F_p) \rightarrow \{0,1\}$ is chosen randomly.
\end{conjecture}

\section*{Acknowledgments}
We thank Lajos R\'onyai for suggesting 
the literature references
cited in \expref{Section}{sec:galois}.
We would like to thank Laci Babai and the
editorial team at ToC for many important comments.
\nocite{conf-version}

\bibliographystyle{plain}  
\bibliography{bibliography}

\end{document}